\documentclass[10pt]{article} 

\title{Bridgeland stability conditions on surfaces with curves of negative self-intersection.}
\author{Rebecca Tramel and Bingyu Xia}
\date{}

\usepackage{stix}
\usepackage[all]{xy}
\usepackage{geometry}               
\usepackage{graphicx}
\usepackage{amssymb}
\usepackage{epstopdf}
\usepackage{amsthm, amsmath}
\usepackage{dsfont}
\usepackage{color}
\usepackage{enumerate}
\usepackage{paralist}
\usepackage{tikz}
\usepackage{amscd}
\usepackage{mathrsfs}
\usetikzlibrary{calc}
\usetikzlibrary{matrix}
\usetikzlibrary{arrows}
\theoremstyle{definiton}
\newtheorem{thm}{Theorem}[section]
\newtheorem{defn}[thm]{Definition}
\newtheorem{prop}[thm]{Proposition}

\newtheorem{lem}[thm]{Lemma}

\newcommand{\BQ}{\mathbb{Q}}
\newcommand{\BA}{\mathbb{A}}
\newcommand{\BC}{\mathbb{C}}
\newcommand{\BP}{\mathbb{P}}

\newcommand{\CO}{\mathcal{O}}

\newcommand{\BL}{\mathbb{L}}
\newcommand{\CA}{\mathcal{A}}
\newcommand{\BR}{\mathbb{R}}
\newcommand{\CB}{\mathcal{B}}

\newcommand{\CT}{\mathcal{T}}
\newcommand{\CF}{\mathcal{F}}
\newcommand{\CC}{\mathcal{C}}

\newcommand{\calD}{\mathcal{D}}
\newcommand{\CE}{\mathcal{E}}

\newcommand{\ch}{\rm ch}
\newcommand{\Coh}{\rm Coh}

\newcommand{\cx}{^{\cdot}}
\newcommand{\CDb}{\mathcal{D}^b}
\newcommand{\Ext}{{\rm Ext}}
\newcommand{\Hom}{{\rm Hom}}
\newcommand{\Imag}{{\rm Im}}
\newcommand{\Real}{{\rm Re}}
\newcommand{\Per}{^{-1}{\rm Per}}

\newcommand{\oks}{\mathcal{O}_{C}(k)[1]}
\newcommand{\okss}{\mathcal{O}_{C}(k)[2]}
\newcommand{\okp}{\mathcal{O}_{C}(k+1)}
\begin{document}

\maketitle

\begin{abstract}
Let $X$ be a smooth complex projective variety. In 2002, \cite{Bridgeland} defined a notion of stability for the objects in $\CDb(X)$, the bounded derived category of coherent sheaves on $X$, which generalized the notion of slope stability for vector bundles on curves. There are many nice connections between stability conditions on $X$ and the geometry of the variety.

We construct new stability conditions for surfaces containing a curve $C$ whose self-intersection is negative. We show that these stability conditions lie on a wall of the geometric chamber of ${\rm Stab}(X)$, the stability manifold of $X$. We then construct the moduli space $M_{\sigma}(\CO_X)$ of $\sigma$-semistable objects of class $[\CO_X]$ in $K_0(X)$ after wall-crossing.

\end{abstract}

\section{Introduction}

Let $X$ be a smooth projective surface, and $\CDb(X)$ be the bounded derived category of coherent sheaves on $X$. Following \cite{AB} we can define certain Bridgeland stability conditions on $X$ by choosing an ample class $\omega$ and another class $B$ in ${\rm NS}(X)$. In \cite{K3} and \cite{AB}, it is shown that these stability conditions lie inside a complex manifold called ${\rm Stab}(X)$, inside an open subset called the geometric chamber. All skyscraper sheaves are stable with respect to these stability conditions. 

It is a natural question to vary such a stability condition continuously within ${\rm Stab}(X)$ and determine at which points some skyscraper sheaves fail to be stable. In other words, we search for walls to the geometric chamber of ${\rm Stab}(X)$. Furthermore, if we consider $M_{\sigma}([\CO_x])$, the moduli space of $\sigma$-stable objects of class $[\CO_x]$, then inside the geometric chamber, $M_{\sigma}([\CO_x]) \cong X$. It is interesting to consider what $M_{\sigma}([\CO_x])$ is after wall-crossing. This question has been considered in \cite{Toda14} and in \cite{K3}. In \cite{Toda14}, the author shows that there is a correspondence between wall-crossing and the minimal model program. He shows that contractions of curves of self-intersection $-1$ can be realized as wall-crossing in ${\rm Stab}(X)$. That is, if $f\colon X \rightarrow Y$ is a birational map contracting a $-1$ curve on $X$, then there is a wall of the geometric chamber such that, after crossing, $M_{\sigma}([\CO_x]) \cong Y$. 

Here we vary the choice of ample divisor $\omega$ until it becomes nef. That is, there is a curve $C$ on $X$ whose intersection with this nef divisor is $0$. We consider the case in which this curve $C \cong \BP^1$ on $X$ such that $C^2=-n$ where $n \geq 2$. In Section \ref{sec:stab}, we construct a wall in the geometric chamber corresponding to the curve $C$, at which the points of $C$ become strictly semistable. 

Given a nef divisor $H$ such that $H \cdot C=0$ and $H \cdot C'>0$ for all curves $C' \not\subseteq C$, and a divisor class $\beta$ such that $H \cdot \beta=0$, we construct a central charge $$Z_{H,\beta}(E\cx)=-\ch_2(E\cx)+\beta \cdot \ch_1(E\cx) +z \, \ch_0(E\cx)+i \, H\cdot \ch_1(E).$$ We construct a heart of a bounded t-structure $\CB_{H,k}^{-\Imag(z)}$ by tilting $\Coh(X)$ twice. 

\newtheorem*{thm:stab}{Theorem \ref{thm:stab}}
\begin{thm:stab}
	The pair $(Z_{H,\beta},\CB_{H,k}^{-\Imag(z)})$ define a stability condition on $\CDb(X)$ when $k$ is chosen so that $k+\frac{n}{2}<\beta \cdot C < k+\frac{n}{2}+1$ and $\Real(z)+\frac{\Imag(z)^2}{H^2} > -\frac{\beta^2}{2}$.
\end{thm:stab}
We further show that we can study wall-crossing by showing this stability condition satisfies the support property \ref{defn:support}.

\newtheorem*{thm:supportbr}{Theorem \ref{thm:supportbr}}
\begin{thm:supportbr}
	The central charge $Z_{H,\beta}$ satisfies the support property for Bridgeland semistable objects in $\CB_{H,k}^{-\Imag(z)}$. 
\end{thm:supportbr}

In Sections \ref{sec:moduli} and \ref{sec:reducedness}, we study the moduli space $M_{\sigma}([\CO_x])$ of stable objects of class $[\CO_x]$ after crossing this wall. We show the following. 

\newtheorem*{thm:mod}{Theorem \ref{thm:mod}}
\begin{thm:mod}
	There is an isomorphism $X \sqcup_C \BP^{n-1} \rightarrow M_{\sigma}([\CO_x])$, where $C$ is embedded in $\BP^{n-1}$ as a rational normal curve. 
\end{thm:mod}

This generalizes the results of \cite{Toda13} for $n=1$ and \cite{K3} for $-2$ curves on K3 surfaces. For $n\geq 3$ this space is reducible, and is the first example in the study of Bridgeland stability in which wall-crossing produces a more complicated moduli space. 

\section{Background}
\label{background}
\label{sec:back}

Let $X$ be a smooth projective surface, and let $\CDb(X)$ be the bounded derived category of coherent sheaves on $X$. In this section, our goal is to recall a notion of stability for objects in $\CDb(X)$ defined in \cite{Bridgeland}, and describe some properties of this definition of stability which will be important in the subsequent sections. 

\begin{defn}
	\label{def:heart}
	A heart of a bounded t-structure is a full additive subcategory $\CA$ of $\CDb(X)$ satisfying 
	\begin{enumerate}
		\item $\Hom^i(A,B)=0$ for $i<0$ and $A,B \in \CA$.
		\item Objects in $\CDb(X)$ have filtrations by cohomology objects in $\CA$. That is, for all nonzero $E\cx \in \CDb(X)$, there is a sequence of exact triangles
		\begin{center}
			\begin{tikzpicture}
			\node at (0,0) {$0=E_0\cx$};
			\node at (2,0) {$E_1\cx$};
			\node at (4,0) {$E_2\cx$};
			\node at (6,0) {$\cdots$};
			\node at (8,0) {$E_{n-1}\cx$};
			\node at (10,0) {$E_n\cx=E\cx$};
			\node at (1,-1) {$A_1\cx$};
			\node at (3,-1) {$A_2\cx$};
			\node at (9,-1) {$A_n\cx$};
			\draw[->] (8.5,0) to (9.3,0);
			\draw[->] (.7,0) to (1.7,0);
			\draw[->] (2.3,0) to (3.5,0);
			\draw[->] (4.3,0) to (5.5,0);
			\draw[->] (6.4,0) to (7.2,0);
			\draw[->] (2,-.2) to (1.2,-.8);
			\draw[->] (10,-.2) to (9.2,-.8);
			\draw[->] (4,-.2) to (3.2,-.8);
			\draw[->, dashed] (.8,-.8) to (0,-.2);
			\draw[->, dashed] (2.8,-.8) to (2,-.2);
			\draw[->, dashed] (8.8,-.8) to (8,-.2);
			\end{tikzpicture}
		\end{center}
		such that $A_i[-k_i] \in \CA$ for integers $k_1>\cdots > k_n$.  
	\end{enumerate}
\noindent The cohomology objects $A_i[-k_i]$ of $E\cx$ in the heart $\CA$ are denoted by $H_{\CA}^{k_i}(E\cx)$.
\end{defn}

It is easy to check that if $\CA$ is a heart of a bounded t-structure in $\CDb(X)$, then $\CA$ is abelian.

\begin{defn}
	\cite[Proposition 5.3]{Bridgeland}
	\label{def:stab}
	A Bridgeland stability condition is a pair $\sigma=(Z, \CA)$ where $Z \colon K_0(\CDb(X)) \rightarrow \BC$ is a group homomorphism and $\CA$ is a heart of a bounded t-structure. The pair must further satisfy that 
	\begin{enumerate}
		\item $Z(\CA \setminus \{0\}) \subseteq \{re^{i\pi \phi} \ \vert \ r>0, \ 0< \phi \leq 1 \}$. Define the phase of $0\neq E \in \CA$ to be $\phi(E):=\phi$. We say $E\in \CA$ is $Z$-semistable if for all nonzero subobjects $F \in \CA$ of $E$, $\phi(F) \leq \phi(E)$. $E$ is $Z$-stable if for all nonzero subobjects $F \in \CA$ of $E$, $\phi(F)<\phi(E)$. 
		\item The objects of $\CA$ have Harder-Narasimhan filtrations with respect to $Z$. That is, for every $E \in \CA$ there is a unique sequence of inclusions $$0=E_0 \subseteq E_1 \subseteq \cdots \subseteq E_{n-1} \subseteq E_n=E$$ such that the successive quotients $E_{i}/E_{i-1}$ are $Z$-semistable, and the phases $\phi(E_1/E_0)> \phi(E_2/E_1) > \cdots > \phi(E_{n-1}/E_{n-2})>\phi(E_{n}/E_{n-1})$.
	\end{enumerate}
\end{defn}

Consider the set of all stability conditions on $X$, denoted ${\rm Stab}(X)$. We will place a restriction on the stability conditions we consider. Recall that there is an Euler pairing on $K(X)$, defined by $\chi(E,F)=\sum_i (-1)^i{\rm dim} \ \Hom^i(E,F)$. We will restrict to stability conditions which factor through the quotient $\mathcal{N}(X)$ of $K(X)$ by the kernel of the this pairing. These are called numerical stability conditions. The set of all such stability conditions is denoted ${\rm Stab}_{\mathcal{N}}(X)$, or simply ${\rm Stab}(X)$. The following theorem says that under this restriction, the set of stability conditions is in fact a complex manifold. 

\begin{thm}\cite[Corollary 1.3]{Bridgeland}
	\label{thm:manifold}
	For each connected component $\Sigma \subseteq {\rm Stab}_{\mathcal{N}}(X)$, there is a subspace $V(\Sigma) \subseteq \Hom(\mathcal{N}(X),\BC)$ and a local homeomorphism $Z\colon \Sigma \rightarrow V(\Sigma)$ which maps a stability condition to its central charge. In particular, $\Sigma$ is a finite-dimensional complex manifold. 
\end{thm}

Given a stability condition on $X$, we would like to be able to deform the stability condition in ${\rm Stab}(X)$ and study how the set of stable objects changes and as the stability condition changes. In order to study such deformations, we will need to require that the stability conditions we study have a sort of continuity property called the support property. By \cite[Proposition B.4]{BM2011}, this is equivalent to the stability condtion being full, as defined in \cite[Definition 4.2]{K3}. 
\begin{defn}
	\label{defn:support}
	A stability condition $\sigma=(Z,\CA)$ satisfies the support property \cite[Section 1.2]{KS} if there exists a constant $C>0$ so that for all $Z$-stable $E\cx \in \CDb(X)$, $$\frac{|Z(E\cx)|}{||E\cx||}>C.$$
\end{defn}

Now let us consider only stability conditions in ${\rm Stab}(X)$ with the support property. Fix a primitive class $[E\cx]$ of objects in $K(\CDb(X))$. Then \cite[Section 9]{K3} shows that ${\rm Stab}(X)$ has a wall and chamber structure. That is, ${\rm Stab}(X)$ decomposes into open subsets $U$ called chambers, $U$, and codimension one closed submanifolds $W$. If $\sigma$ is a stability condition in chamber $U$ and $E\cx$ is a $\sigma$-stable objects of class $[E\cx]$, then $E\cx$ remains stable for all other stability conditions in $U$. That is, stable objects of class $[E\cx]$ may only destabilize along walls $W$.

Let $x \in X$. Consider the class $[\CO_x]$ of the skyscraper sheaf at $x$ in $K(X)$. There is a special set of stability conditions called geometric stability conditions, constructed by \cite{K3} for K3 surfaces, and by \cite{AB} for all smooth projective surfaces. These are stability conditions for which all skyscraper sheaves are stable. The chamber of ${\rm Stab}(X)$ containing these stability conditions is called the geometry chamber. The goal of this paper is to deform these stability conditions to construct a wall in ${\rm Stab}(X)$ for surfaces $X$ which contain a curve of negative self-intersection, and to describe the moduli space of stable objects of class $[\CO_x]$ across this wall. 

We now describe the construction of geometric stability conditions from \cite{AB}, as this will be the starting point for our later construction. First, we must construct a heart of a bounded t-structure.

\begin{defn}
	A torsion pair in a heart $\CA$ is a pair $(\CT , \CF)$ of full additive subcategories of $\CA$ such that 
	\begin{enumerate}
		\item If $T \in \CT$ and $F \in \CF$, then $\Hom(T,F)=0$.
		\item For all $E \in \CA$ there is an object $T \in \CT$ and $F \in \CF$ so that the sequence $0 \rightarrow T \rightarrow E \rightarrow F \rightarrow 0$ is exact.
	\end{enumerate}
\end{defn}

Given a torsion pair $(\CT, \CF)$ in $\CA$, we can construct a new heart of a bounded t-structure $$\CA^{\#}=\{E\cx \in \CDb(X) \, | \, H^0_{\CA}(E\cx) \in \CT, \, H^{-1}_{\CA}(E\cx)\in \CF, \, H^{i}_{\CA}(E\cx)=0 \ {\rm for} \ i\neq 0,-1\}.$$ This new heart is called a tilt of $\CA$.

We can define stability on a surface $X$ on a tilt of the standard heart $\Coh(X)$. This is the tilt at slope \cite[Lemma 6.1]{K3}. First, we fix an ample divisor $H$ on $X$. The slope of a nonzero sheaf $E \in \Coh(X)$ is 
$$\mu_H(E) = \left\{
\begin{array}{lr}
\frac{H \cdot \ch_1(E)}{\ch_0(E)} &  E \ {\rm torsion-free} \\
\infty &   E \ {\rm torsion} \\
\end{array}
\right.$$
\begin{defn}
	\label{def:stablesheaf}
	A sheaf $E$ is $\mu_H$-stable if for all subobjects $0 \neq F \subseteq E$, $\mu_H(F)<\mu_H(E)$. $E$ is $\mu_H$-semistable if for all subobjects $0 \neq F \subseteq E$, $\mu_H(F) \leq \mu_H(E)$. 
\end{defn}
Note that it would be equivalent to define $E$ to be $\mu_H$-stable if for all quotients $E \twoheadrightarrow G$, $\mu_H(E)< \mu_H(G)$. 

Fix a number $a \in \BR$.
$$\CT_H^a:=\{T \in \Coh(X) \, | \, {\rm for \ all} \ T \twoheadrightarrow S, \ \mu_H(S)>a\}. $$
$$\CF_H^a:=\{F \in \Coh(X) \, | \, {\rm for \ all} \ G \hookrightarrow F, \ \mu_H(G)\leq a\}. $$ Note that all torsion sheaves and $\mu_H$-semistable sheaves of slope greater than $a$ lie in $\CT^a$, and all $\mu_H$-semistable sheaves of slope smaller than or equal to $a$ lie in $\CF^a$. 
\begin{lem}
	\label{lem:torsslope}
	$(\CT_H^a, \CF_H^a)$ is a torsion pair in $\Coh(X)$. 
\end{lem}

Following this lemma, let $\CA_H^0$ be the tilt of $\Coh(X)$ at the torsion pair $(\CT_H^0,\CF_H^0)$. The following is due to \cite{AB} and to \cite{K3} in the case that $X$ is a K3 surface.

\begin{prop}
	\label{prop:tilt}
	Choose a class $\beta \in NS_{\BR}(X)$. The pair $\sigma_{H,\beta}=(Z_{H,\beta},\CA_H^0)$ is a Bridgeland stability condition on $X$, where $Z_{H,\beta}(E\cx)=-\int_X\ch(E\cx)e^{\beta+iH}$.
\end{prop}

\section{Construction of a heart from a nef divisor}
\label{sec:stab}

Let $X$ be a smooth projective surface which contains a smooth, rational curve $C$ whose self-intersection is negative. Say $C^2=-n$ where $n \geq 2$. An example of such a surface is the Hirzebruch surface constructed as the projectivisation of the sheaf $\CO_{\BP^1}\oplus \CO_{\BP^1}(-n)$ on $\BP^1$. We now begin construction of a wall to the geometric chamber of ${\rm Stab}(X)$. Choose a nef divisor $H$ on $X$ satisfying that $C \cdot H=0$ and that $C'\cdot H>0$ for all curves $C'$ not contained in $C$.

For $E \in \Coh(X)$ torsion-free define
$$\nu_H(E) = \left\{
\begin{array}{lr}
\frac{H \cdot \ch_1(E)}{\ch_0(E)} &  E \ {\rm torsion-free} \\
\infty &   E \ {\rm torsion} \\
\end{array}
\right.$$
This slope function is the generalisation of the construction in Section \ref{background} to the case in which $H$ is nef. Define $\nu_H$ stability as $\mu_H$ stability was defined in Section \ref{background}. Fix $a \in \BR$, and define the following subcategories of ${\rm Coh}(X)$.
$$\CT_H^a=\{ T \in \Coh(X) \ \vert \ \nu_H(S)>a \ {\rm for \ all} \ T \twoheadrightarrow S \}.$$
$$\CF_H^a=\{F \in \Coh(X) \ \vert \ \nu_H(G)\leq a \ {\rm for \ all} \ G \hookrightarrow F\}.$$ By Lemma \ref{lem:torsslope} these two subcategories of $\Coh(X)$ are a torsion pair.

Let $$\CA_H^a:=\{E\cx\in \CDb(X) \ \vert \ H^0(E\cx) \in \CT_H^a, \ H^{-1}(E\cx) \in \CF_H^a, \ H^i(E\cx)=0 \ {\rm if} \ i\neq 0,-1 \}.$$ Unlike in Proposition \ref{prop:tilt}, $H$ is nef, and so this will not necessarily form part of a Bridgeland stability condition on $X$. We will instead tilt this heart again, at a torsion pair constructed by considering sheaves supported on the curve $C$.

If we consider now the sheaves $\CO_C(i)$, the twists of the structure sheaf of $C$, we see that hese are torsion sheaves on $X$, and so each has slope $\infty$ for all choices of $H$. This means that all such sheaves lie in $\CT_H^a$, and so in $\CA_H^a$. Recall that for $\mathcal{S} \subseteq \CDb(X)$, $\langle \mathcal{S} \rangle$ is notation for the extension closure of $\mathcal{S}$. That is, $\langle \mathcal{S} \rangle$ is the smallest subcategory of $\CDb(X)$ closed under taking extensions of objects in $\mathcal{S}$. We will now define the following subcategories of $\CA_H^a$.

The first subcategory we define is
$$\CF_{H,k}^a=\langle \CO_C(i) \ \vert \ i\leq k \rangle.$$
We then define another subcategory to be the left orthogonal to $\CF_{H,k}^a$. That is,
$$\CT_{H,k}^a=\{ E\cx \in \CA_H^a \ \vert \ {\rm Hom}(E\cx,\CO_C(i))=0 \ {\rm for} \ i \leq k\}.$$ 

\begin{lem}
	\label{lem:dechoms}
	If there is a sequence of inclusions in $\CA_H^{-\Imag(z)}$, say $$\cdots \hookrightarrow S_{i}\cx \hookrightarrow S_{i-1}\cx \hookrightarrow \cdots \hookrightarrow S_1\cx\hookrightarrow S_0\cx$$ whose quotients lie in $\CF_{H,k}^{-\Imag(z)}$, then for $i \gg 0$, $S_i\cx \cong S_{i-1}\cx$
\end{lem}

\begin{proof}
	Suppose there is a sequence of inclusions 
	\begin{equation}
	\label{eq:tp1}
	\cdots \hookrightarrow S_{i+1}\cx \hookrightarrow S_{i}\cx \hookrightarrow \cdots \hookrightarrow S_1\cx\hookrightarrow S_0\cx
	\end{equation}
	such that for all $i$, $S_i\cx \in \CA_{H}^{-\Imag(z)}$, and the quotient $F_i$ of the map $S_{i+1}\cx \hookrightarrow S_{i}\cx$ lies in  $\CF_{H,k}^{-\Imag(z)}$. First note that if we take the long exact sequence of cohomology, for every $i$, $H^{-1}(S_i) \cong H^{-1}(S_{i+1})$, and there is a corresponding sequence of sheaves $$\cdots \hookrightarrow H^0(S_{i+1}\cx) \hookrightarrow H^0(S_i\cx) \hookrightarrow \cdots \hookrightarrow H^0(S_1\cx) \hookrightarrow H^0(S_0\cx)$$ whose quotients are the same sheaves $F_i$. Hence it is enough to prove that $\ref{eq:tp1}$ stabilizes when the $S_i$ in (\ref{eq:tp1}) are sheaves in $\CT_H^{-\Imag(z)}$.
	
	Furthermore, every $F_i \in \CF_{H,k}^{-\Imag(z)}$ has a nonzero surjective morphism in $\Coh(X)$ to $\CO_C(l_i)$ for some $l_i \leq k$. Let $S_i^{(1)}$ be the kernel of the composition $S_i \rightarrow F_i \rightarrow \CO_C(l_i)$ We can see via the octahedral axiom that there is an exact sequence of sheaves $$0 \rightarrow S_i^{(1)} \rightarrow S_i \rightarrow \CO_C(l_i) \rightarrow 0.$$ 
	
	The quotient $F_i^{(1)}$ of the map $S_{i+1} \rightarrow S_i$ fits into an exact sequence $$0 \rightarrow F_i^{(1)} \rightarrow F_i \rightarrow \CO_C(l_i) \rightarrow 0.$$ This implies that $F_i^{(1)} \in \CF_{H,k}^{-\Imag(z)}$. Since $F_i \in \CF_{H,k}^{-\Imag(z)}$, $\ch_1(F_i)=m[C]$ for some $m \in \mathbb{N}$. Hence $\ch_1(F_i^{(1)})=(m-1)[C]$. We can now apply this process to the map $S_{i+1} \rightarrow S_i^{(1)}$ and repeat until we have a sequence $$S_{i+1} \hookrightarrow S_i^{(m-1)} \hookrightarrow \cdots \hookrightarrow S_i^{(1)} \hookrightarrow S_i,$$ all of whose quotients are sheaves of the form $\CO_C(l_i^{(j)})$ for some $l_i^{(j)} \leq k$. By applying this process to (\ref{eq:tp1}), we can assume each quotient $F_i$ in (\ref{eq:tp1}) is in fact $\CO_C(l_i)$ for some $l_i \leq k$. 
	
	Consider the exact sequence $$0 \rightarrow S_{i+1} \rightarrow S_{i} \rightarrow \CO_C(l_{i}) \rightarrow 0.$$ Since $l_i \leq k$, we can compute $\Hom(\CO_C(l_{i}),\CO_C(k)) \cong \BC^{k-l_i+1}$. Furthermore, $\Ext^1(\CO_C(l_i),\CO_C(k)) \cong \mathcal{H}^1(X,\CO_C(k) \otimes \CO_C(l_i)^{\vee})$. As there is an exact sequence $$0 \rightarrow \CO_X(-C)(l_i) \rightarrow \CO_X(l_i) \rightarrow \CO_C(l_i) \rightarrow 0$$ in $X$, we can compute $\CO_C(l_i)^{\vee}$ in $\CDb(X)$ as the complex $\CO_C(-l_i) \rightarrow \CO_C(-n-l_i)$. There are no morphisms between the two sheaves in this complex, hence we have $\mathcal{H}^1(X,\CO_C(k) \otimes \CO_C(l_i)^{\vee}) \cong \mathcal{H}^1(X,\CO_C(k-l_i)) \oplus \mathcal{H}^0(X,\CO_C(k-l_i-n))$. Hence if $k-l_i-n \geq 0$, $\Ext^1(\CO_C(l_i),\CO_C(k)) \cong \BC^{k-l_i-n+1}$, otherwise it is $0$. By a similar calculation, if $k-l_i-n<-1$, $\Ext^2(\CO_C(l_i),\CO_C(k)) \cong \BC^{l_i-k+n-1}$, otherwise it is zero.
	
	In particular, this means that either $\Ext^1(\CO_C(l_i),\CO_C(k)) \cong 0$ or $\Ext^2(\CO_C(l_i),\CO_C(k)) \cong 0$. Suppose first that $\Ext^1(\CO_C(l_i),\CO_C(k)) \cong 0$. Then taking the long exact sequence of cohomology, we see there is an exact sequence $$0 \rightarrow \Hom(\CO_C(l_i),\CO_C(k)) \rightarrow \Hom(S_i,\CO_C(k)) \rightarrow \Hom(S_{i+1},\CO_C(k)) \rightarrow 0.$$ Since $\Hom(\CO_C(l_i),\CO_C(k)) \neq 0$, this means that ${\rm dim} \ \Hom(S_i,\CO_C(k))>{\rm dim} \ \Hom(S_{i+1},\CO_C(k))$.
	
	Now suppose that $\Ext^2(\CO_C(l_i),\CO_C(k)) \cong 0$. Then again applying $\Hom(-,\CO_C(k))$ to the exact sequence $$0 \rightarrow S_{i+1} \rightarrow S_{i} \rightarrow \CO_C(l_{i}) \rightarrow 0$$ we see that $\Ext^2(S_i,\CO_C(k)) \cong \Ext^2(S_{i+1},\CO_C(k))$ and there is an exact sequence 
	\begin{gather*} 
	0 \rightarrow \Hom(\CO_C(l_i),\CO_C(k)) \rightarrow \Hom(S_i,\CO_C(k)) \rightarrow \Hom(S_{i+1},\CO_C(k)) \rightarrow \Ext^1(\CO_C(l_i),\CO_C(k)) \rightarrow \\
	\rightarrow \Ext^1(S_i,\CO_C(k)) \rightarrow \Ext^1(S_{i+1},\CO_C(k)) \rightarrow 0.
	\end{gather*}
	The sequence above is exact, so the alternating sum of the dimensions is $0$. That is,
	\begin{gather*}
	{\rm dim} \ \Hom(S_i,\CO_C(k))- {\rm dim} \ \Hom(S_{i+1},\CO_C(k)) = n-{\rm dim} \ \Ext^1(S_i,\CO_C(k)) - {\rm dim} \ \Ext^1(S_{i+1},\CO_C(k)).
	\end{gather*}
	Since the map $\Ext^1(S_i,\CO_C(k)) \rightarrow \Ext^1(S_{i+1},\CO_C(k))$ is surjective, we can say that ${\rm dim} \ \Ext^1(S_i,\CO_C(k)) > {\rm dim} \ \Ext^1(S_{i+1},\CO_C(k))$. Hence we see that in this case as well, ${\rm dim} \ \Hom(S_i,\CO_C(k))> {\rm dim} \ \Hom(S_{i+1},\CO_C(k))$. As these dimensions decrease when $i$ increases, we see that the sequence must terminate. 
\end{proof}

\begin{lem} 
	\label{tiltpairlem}
	The pair $(\CT_{H,k}^a,\CF_{H,k}^a)$ form a torsion pair in $\CA_H^a$. 
\end{lem}

\begin{proof}
	We must show that for any $E\cx \in \CA_{H}^a$, there is an exact triangle $$T\cx \rightarrow E\cx \rightarrow F$$ such that $T\cx \in \CT_{H,k}^a$ and $F \in \CF_{H,k}^a$. If $\Hom(E\cx, \CF_{H,k}^a)\neq 0$, then let $$\hat{S}_1\cx \rightarrow E\cx \rightarrow F$$ be an exact triangle with $F \in \CF_{H,k}^a$. 
	
	Taking the long exact sequence of sheaf cohomology, there is an exact sequence $$0 \rightarrow H^0(\hat{S}_1\cx) \rightarrow H^0(E\cx) \rightarrow F \rightarrow H^1(\hat{S}_1\cx) \rightarrow 0.$$ Let $F_1$ be the image of the map from $H^0(E\cx)$ to $F$. By composing the morphism from $E\cx$ to $H^0(E\cx)$ wiith this map, we get a morphism from $E\cx$ to $F_1$. Then there is an exact triangle $$S_1\cx \rightarrow E\cx \rightarrow F_1.$$ We will show that $S_1\cx \in \CA_{H}^a$. 

First, if we look at the long exact sequence of sheaf cohomology, we see that $H^{-1}(\hat{S}_1\cx) \cong H^{-1}(E\cx) \in \CF_H^a$. By construction, we also have a short exact sequence $$0 \rightarrow H^0(S_1\cx) \rightarrow H^0(E\cx) \rightarrow F_1 \rightarrow 0.$$ Let $G$ be any quotient of $H^0(S_1\cx)$, fitting into exact sequence $$0 \rightarrow R \rightarrow H^0(S_1\cx) \rightarrow G \rightarrow 0.$$ Then by composing the maps $R \hookrightarrow H^0(S_1\cx) \hookrightarrow H^0(E\cx)$ there is a short exact $$0 \rightarrow G \rightarrow H^0(E\cx)/R \rightarrow F_1.$$ Since $F_1$ is supported on $C$, $\nu_H(G)=\nu_H(H^0(E\cx)/R).$ And since $H^0(E\cx) \in \CT_H^a$, $\nu_H(H^0(E\cx)/R)>a$. Hence $S_1\cx \in \CA_H^a$.
	
	If $\Hom(S_1\cx,\CF_{H,k}^a) \neq 0$, then we can repeat this process, and construct an exact triangle $$S_2\cx \rightarrow S_1\cx \rightarrow F_2$$ with $F_2 \in \CF_{H,k}^a$. If we iterate this process we get a sequence of complexes $S_i\cx \in \CA_H^a$, such that $H^{-1}(S_i\cx) \cong H^{-1}(E\cx)$, and such that there is a descending chain of inclusions
	$$H^0(E\cx) \supseteq H^0(S_1\cx) \supseteq \cdots \supseteq H^0(S_i\cx) \supseteq H^0(S_{i+1}\cx) \supseteq \cdots$$ in $\Coh(X)$.
	
	By Lemma \ref{lem:dechoms}, this chain must terminate. That is, there exists a number $n$ such that for $i \geq n$, $H^0(S_i\cx) \cong H^0(S_{i+1}\cx)$. It follows that $\Hom(S_n\cx, \CF_{H,k}^a)=0$, and $$S_n\cx \rightarrow E\cx \rightarrow F_n$$ is the desired triangle. 
\end{proof}

We now tilt the heart $\CA_H^a$ and define the following heart in $\CDb(X)$:
$$\CB_{H,k}^a:=\{E\cx\in \CDb(X) \ \vert \ H_{\CA_H^a}^0(E\cx) \in \CT_{H,k}^a, \ H_{\CA_H^a}^{-1}(E\cx) \in \CF_{H,k}^a, \ H_{\CA_H^a}^i(E\cx)=0 \ {\rm if} \ i\neq0,-1 \}.$$

\section{Comparison with heart constructed in \cite{Toda13}}

We will now explain how the heart we have constructed compares with the heart in \cite[Section 3.1]{Toda13}. This is not necessary to the construction of our stability condition, it is for the purpose of comparison. We will show that our heart and Toda's coincide when $n=1$ and $a=0$.  

Following \cite[Section 3.1]{Toda13}, let $C$ be a curve on a smooth projective surface $X$ such that $C^2=-1$, and let $f\colon X \rightarrow Y$ be the map contracting this $-1$ curve. Let $H=f^*\omega$ be the pull-back of ample divisor $\omega$ on $Y$. Toda constructs a heart of a bounded t-structure in $\CDb(X)$ as a tilt of $\Per(X/Y)$, the category of perverse sheaves on $X$. This category can itself can be constructed as a tilt of $\Coh(X)$ as in \cite[Lemma 3.1.1]{vdbergh}.

Let $\CC=\{E \in {\rm Coh}X \ \vert \ \BR f_*E=0 \}$. Note that the only sheaf supported on $C$ which lies in $\CC$ is $\CO_C(-1)$. Consider the following torsion pair in $\Coh(X)$.
$$\CT_{-1}=\{E \in \Coh(X) \ \vert \ R^1f_*E=0, \  {\rm Hom}(E,\CC)=0 \}.$$
$$\CF_{-1}=\{E \in \Coh(X) \ \vert \ f_*E=0\}.$$
Then $\Per(X/Y)$ is the tilt of ${\rm Coh}(X)$ at the torsion pair $(\CT_{-1},\CF_{-1})$. That is, 
$$\Per(X/Y)=\{ E\cx \in \CDb(X) \ \vert \ H^0(E\cx) \in \CT_{-1}, \ H^{-1}(E\cx) \in \CF_{-1}, \ H^i(E\cx)=0 \ {\rm if} \ i\neq 0,-1\}.$$

Now define a slope function on $\Per(X/Y)$ as we did for $\Coh(X)$. 
For $E\cx \in \Per(X/Y)$, define 
$$\mu_{f^*\omega}(E\cx) = \left\{
\begin{array}{lr}
\frac{\ch_1(E\cx)\cdot f^*\omega}{\ch_0(E\cx)} &  \ch_0(E\cx) \neq 0\\
\infty &  \ch_0(E\cx)=0, E\cx \neq 0 \\
-\infty & E\cx=0 \\
\end{array}
\right.$$

We will now tilt the category of perverse sheaves at slope, as we did for coherent sheaves before. Let
$$^{-1}\CT_{f^*\omega}=\{T\cx \in \Per(X/Y) \ \vert \ \mu_{f^*\omega}(S\cx)>0 \ {\rm for \ all} \ T\cx \twoheadrightarrow S\cx \},$$
$$^{-1}\CF_{f^*\omega}=\{F\cx \in \Per(X/Y) \ \vert \ \mu_{f^*\omega}(E\cx)\leq 0 \ {\rm for \ all} \ E\cx \hookrightarrow F\cx \}.$$
Toda then is able to define a stability condition on the following heart, where $H^i_{Per}$ refers to cohomology with respect to the heart $\Per(X/Y)$:
$$\CB_{f^*\omega}=\{E\cx \in \CDb(X) \ \vert \ H^0_{Per}(E\cx) \in ^{-1}\CT_{f^*\omega}, \ H^{-1}_{Per}(E\cx) \in ^{-1}\CF_{f^*\omega}, \ H^i_{Per}(E\cx)=0 \ {\rm if} \ i\neq 0,-1\}.$$

\begin{lem}
	\label{lem:freeparts}
	For any ample divisor $\omega$ on $Y$, $\CF_{f^*\omega,-1}^0=\CF_{-1}$.
\end{lem}
\begin{proof}
	First, since $\CO_C(i)$ has no global sections for $i<0$, $f_*\CO_C(i)=0$ when $i<0$. Now suppose $E$ is a sheaf in $\CF_{-1}$, that is $f_*E=0$. Then since $X\setminus C \cong Y \setminus P$, the support of $E$ must be contained in $C$. Specifically, $E$ must be a sheaf on $C$ with no global sections. This implies $f_*E=0$ and $E \in \CF_{f^*\omega,-1}^0$. 
\end{proof}

\begin{lem}
	\label{lem:slopes}
	For any $E\cx \in \Per(X/Y)$ such that $H^0(E\cx)\neq 0$, $\mu_{f^*\omega}(E\cx)=\nu_{f^*\omega}(H^0(E\cx))$.
\end{lem}
\begin{proof}
	This follows from the fact that $\ch(E\cx)=\ch(H^0(E\cx))-\ch(H^{-1}(E\cx))$. Since $H^{-1}(E\cx)$ is supported on $C$, $\ch_0(H^{-1}(E\cx))=0$ and $\ch_1(H^{-1}(E\cx))\cdot f^*\omega=0$.
\end{proof}

\begin{prop}
	\label{prop:slopes}
	Let $E\cx$ be a perverse sheaf such that $H^0(E\cx) \neq 0$.
	\begin{enumerate}
		\item $H^0(E\cx) \in \CT_{f^*\omega}^0$ if and only if $E\cx \in ^{-1}\CT_{f^*\omega}$. 
		\item $H^0(E\cx) \in \CF_{f^*\omega}^0$ if and only if $E\cx \in ^{-1}\CF_{f^*\omega}$. 
	\end{enumerate}
\end{prop}
\begin{proof}
	Suppose first that $H^0(E\cx)$ is in $\CT_{f^*\omega}^0$. Because perverse sheaves have cohomology only in degrees $-1$ and $0$, for any quotient $S\cx$ of $E\cx$ in $\Per(X/Y)$, $H^0(E\cx)$ surjects onto $H^0(S\cx)$. All quotient sheaves of $H^0(E\cx)$ have positive slope. This implies that $\mu_{f^*\omega}(S\cx)=\nu_{f^*\omega}(H^0(S\cx))>0$, and $E\cx$ is in $^{-1}\CT_{f^*\omega}$. 
	
	Now suppose that $E\cx \in ^{-1}\CT_{f^*\omega}$. Let $H^0(E\cx) \rightarrow S$ be a surjective map of coherent sheaves. Then $S$ is necessarily also in $\CT_{-1}$, that is, $S$ is a perverse sheaf. However, the map $E\cx \rightarrow H^0(E\cx) \rightarrow S$ may not be a surjection in $\Per(X/Y)$. That is, if $P\cx$ is the kernel of the composition, fitting into exact triangle
	\begin{align}
	\label{seq1}
	P\cx \rightarrow E\cx \rightarrow S,
	\end{align}
	it may be that $P\cx$ is not in $\Per(X/Y)$, since $H^0(P\cx)$ need not be in $\CT_{-1}$. We will now construct a perverse sheaf $S'$ such that $\mu_{f^*\omega}(S')=\mu_{f^*\omega}(S)$ and $E\cx \twoheadrightarrow S'$, proving that $\mu_{f^*\omega}(E\cx)>0$. 
	
	Since $H^0(P\cx)$ is a sheaf, there exist sheaves $T \in \CT_{-1}$ and $F \in \CF_{-1}$ so that 
	\begin{align}
	\label{seq2}
	0 \rightarrow T \rightarrow H^0(P\cx) \rightarrow F \rightarrow 0
	\end{align}
	is exact. Further, since $F$ is supported on $C$, $\ch_0(F)=H\cdot \ch_1(F)=0$ and $\nu_{f^*\omega}(T)=\nu_{f^*\omega}(H^0(P\cx))$. There is an injective map of sheaves $T \rightarrow H^0(E\cx)$ composing the injective maps $T \rightarrow H^0(P\cx) \rightarrow H^0(E\cx)$. Let $S'$ be the quotient sheaf of this map, fitting into exact sequence
	\begin{align}
	\label{seq3}
	0 \rightarrow T \rightarrow H^0(E\cx)\rightarrow S'\rightarrow 0.
	\end{align}
	Again, $S'$ is also necessarily a perverse sheaf. We also claim that $\mu_{f^*\omega}(S')=\mu_{f^*\omega}(S)$. The sequence (\ref{seq1}) gives rise to a long exact sequence of sheaves
	$$0 \rightarrow H^{-1}(P\cx) \rightarrow H^{-1}(E\cx)\rightarrow 0 \rightarrow H^0(P\cx)\rightarrow H^0(E\cx)\rightarrow S\rightarrow 0.$$ We can conclude by additivity of chern characters that 
	\begin{align}
	\label{seq4}
	\ch_0(H^0(E\cx))=\ch_0(S)+\ch_0(H^0(P\cx)).
	\end{align}
	Sequence (\ref{seq2}) shows that $\ch_0(H^0(P\cx))=\ch_0(T)$. Thus we can rewrite equation (\ref{seq4}) as 
	\begin{align}
	\label{seq5}
	\ch_0(H^0(E\cx))=\ch_0(S)+\ch_0(T).
	\end{align}
	But taking the long exact sequence of (\ref{seq3}) we have $$\ch_0(H^0(E\cx))=\ch_0(S')+\ch_0(T).$$ Hence, $\ch_0(S)=\ch_0(S')$. Note that equations (\ref{seq4}) and (\ref{seq5}) can also be written for $\ch_1$, to show that $\ch_1(S)=\ch_1(S')$. Thus we have shown that $\mu_{f^*\omega}(S)=\mu_{f^*\omega}(S')$.

	We will now show that the composition $E \cx \rightarrow H^0(E\cx) \rightarrow S'$ is surjective in $\Per(X/Y)$. Let $Q\cx$ be the kernel of the composition $E \cx \rightarrow H^0(E\cx) \rightarrow S'$, fitting into exact triangle 
	$$Q \cx \rightarrow E\cx \rightarrow S'.$$ Note that $E$ surjects onto $S'$ in $\Per(X/Y)$ if and only if $Q\cx \in \Per(X/Y)$. Taking long exact cohomology, $H^{-1}(Q\cx) \cong H^{-1}(E\cx)$ which is in $\CF_{-1}$, and $H^0(Q\cx) \cong T$ which is in $\CT_{-1}$. Then $Q\cx \in \Per(X/Y)$, and so $E\cx \rightarrow S'$ is surjective in $\Per(X/Y)$. This implies $\nu_{f^*\omega}(S)=\mu_{f^*\omega}(S')>0$. 
	
	We will now prove the second statement. Suppose $H^0(E\cx) \in \CF_{f^*\omega}^0$, and $F\cx \rightarrow E\cx$ is an injection in $\Per(X/Y)$ with quotient $G$. There is a long exact cohomology sequence 
	\begin{center}
		\begin{tikzpicture}
		\matrix (m) [matrix of math nodes,column sep=1em, row sep=2em]
		{0 & & H^{-1}(F\cx) & & H^{-1}(E\cx) & & H^{-1}(G\cx) & & H^0(F\cx) & & H^0(E\cx) & & H^0(G\cx) & & 0. \\
			& & & & & & & K & & I & & & & & \\ };
		
		\path[-stealth]
		(m-1-1) edge (m-1-3) 
		(m-1-3) edge (m-1-5) 
		(m-1-5) edge (m-1-7) 
		(m-1-7) edge (m-1-9)
		(m-1-9) edge (m-1-11)
		(m-1-11) edge (m-1-13)
		(m-1-13) edge (m-1-15)
		(m-1-7) edge (m-2-8)
		(m-2-8) edge (m-1-9)
		(m-1-9) edge (m-2-10)
		(m-2-10) edge (m-1-11);
		\end{tikzpicture}
	\end{center}
	Since $K$ is a quotient of $H^{-1}(G)$ it must be supported on $C$. This implies that $\ch_0(K)=H\cdot \ch_1(K)=0$. Then $\mu_{f^*\omega}(F\cx)=\nu_{f^*\omega}(H^0(F\cx))=\nu_{f^*\omega}(I) \leq 0$ since $I$ is a subsheaf of $H^0(E\cx)$. 
	
	Now suppose $E\cx \in ^{-1}\CF_{f^*\omega}$. Let $S \rightarrow H^0(E\cx)$ be an injective morphism of sheaves. We will construct a perverse sheaf $R\cx$ which injects into $E\cx$ so that $\mu(Q\cx)=\mu(E\cx)$. Since $S$ is a sheaf, it fits into an exact sequence $$0 \rightarrow T \rightarrow S \rightarrow F \rightarrow 0$$ where $T \in \CT_{-1}$ and $F \in \CF_{-1}$. Since $\ch_0(F)=H\cdot \ch_1(F)=0$, $\nu_{f^*\omega}(T)=\nu_{f^*\omega}(S)$. Composing the injective maps $T \rightarrow S \rightarrow H^0(E\cx)$, we get an exact sequence $$0 \rightarrow T \rightarrow H^0(E\cx) \rightarrow Q \rightarrow 0$$ for some $Q \in {\rm Coh}(X)$. $H^0(E\cx)$ is a perverse sheaf, and so $Q \in \CT_{-1}$ is a perverse sheaf. 
	
	We have morphisms $E\cx \rightarrow H^0(E\cx) \rightarrow Q$. Let $R\cx$ be the kernel of the composition, fitting into exact triangle
	$$R\cx \rightarrow E\cx \rightarrow Q.$$ Taking the long exact cohomology sequence we see $H^{-1}(R\cx) \cong H^{-1}(E\cx) \in \CF_{-1}$ and $H^0(R\cx) \cong T \in \CT_{-1}$, so $R\cx \in \Per(X/Y)$. This means that $R\cx \rightarrow E\cx$ is an injective morphism in $\Per(X/Y)$. And so $0 \geq \mu_{f^*\omega}(R\cx)=\nu_{f^*\omega}(T)=\nu_{f^*\omega}(S)$.   
	
\end{proof}

Although Proposition \ref{prop:slopes} did not address perverse sheaves $E\cx$ for which $H^0(E\cx)=0$, it is easy to see that the slope $\mu_{f^*\omega}(E\cx)$ of such a perverse sheaf is $\infty$, and that in this case $E\cx \in ^{-1}\CT_{f^*\omega}$.

\begin{thm}
	\label{thm:tiltorder}
	$\CB_{f^*\omega}=\CB_{f^*\omega,-1}^0$.
\end{thm}
\begin{proof}
	It suffices to show that $\CB_{f^*\omega,-1}^0 \subset \CB_{f^*\omega}$, since each is a heart of a bounded t-structure. Suppose $E\cx \in \CB_{f^*\omega,-1}^0$. Then there is an exact triangle $$F[1] \rightarrow E\cx \rightarrow T\cx$$ where $F$ is in $\CF_{f^*\omega,-1}^0$ and $T\cx \in \CT_{f^*\omega,-1}^0$. By Lemma \ref{lem:freeparts}, $\CF_{f^*\omega,-1}^0=\CF_{-1}$. The sheaves in $\CF_{-1}$ are torsion, so $\mu_{f^*\omega}(F[1])=\infty$ and $F[1] \in ^{-1}\CT_{f^*\omega}\subset \CB_{f^*\omega}$. 
	
	It remains to show that $T\cx \in \CB_{f^*\omega}$. Since $T\cx \in \CT_{f^*\omega,-1}^0$, it is contained in $\CA_{f^*\omega}^0$. This means there is an exact triangle $$H^{-1}(T\cx)[1] \rightarrow T\cx \rightarrow H^0(T\cx)$$ with $H^0(T\cx) \in \CT_{f^*\omega}^0$ and $H^{-1}(T\cx) \in \CF_{f^*\omega}^0$. We will now show that $H^{-1}(T\cx)$ and $H^0(T\cx)$ also lie in $\CB_{f^*\omega}$. 
	
	There is an exact sequence $$0 \rightarrow S_{-1} \rightarrow H^{-1}(T\cx) \rightarrow R_{-1} \rightarrow 0$$ with $S_{-1} \in \CT_{-1}$ and $R_{-1} \in \CF_{-1}$. Clearly $R_{-1}[1] \in \CB_{f^*\omega}$. Since $S_{-1}$ is a subsheaf of $H^{-1}(T\cdot)$, $S_{-1} \in \CT_{-1} \cap \CF_{f^*\omega}$, and so $S_{-1}[1] \in \CB_{f^*\omega}$ as well. Thus $H^{-1}(T\cx)[1] \in \CB_{f^*\omega}$. 
	
	Similarly, there is an exact sequence  $$0 \rightarrow S_{0} \rightarrow H^{0}(T\cx) \rightarrow R_{0} \rightarrow 0$$ with $S_{0} \in \CT_{-1}$ and $R_{0 } \in \CF_{-1}$. We know there are no nonzero maps $T\cx \rightarrow R_0$. So then if $R_0$ is nonzero, we get an exact triangle $H^{-1}(T\cx)[2] \rightarrow C\cx \rightarrow S_0[1]$ by the octahedral axiom, where $C\cx$ is the cone of the $0$ map from $T\cx \rightarrow R_0$. Taking the long exact sequence of cohomology we find that $H^0(C\cx)=0$. But since this is the cone of the zero morphism, $H^0(C\cx)\cong R_0$. This shows that $H^0(T\cx) \cong S_0 \in \CT_{-1}\cap \CT_{f^*\omega}$. Thus $H^0(T\cx) \in \CB_{f^*\omega}$. 
\end{proof}

\section{Central charge corresponding to $\CB_{H,k}^a$}

Suppose now that $C$ is a curve on the smooth projective surface $X$ with $C^2=-n$. Suppose further that there is a nef divisor $H$ on $X$ so that $C \in H^{\perp}$, but $H\cdot C'>0$ for all curves $c' \subseteq X$ so that $C' \not\subseteq C$. 

Let $z \in \BC$ and let $\beta \in {\rm NS}_{\BR}(X)$ so that $\beta \cdot H=0$. We want to define a central charge $$Z_{H,\beta}(E)=-\ch_2(E)+iH\ch_1(E)+\beta \ch_1(E)+z \ch_0(E)$$ on $\CDb(X)$. We will now show that the pair $(Z_{H,\beta}, \CB_{H,k}^{-\Imag(z)})$ is a stability condition if $k$ and $z$ satisfy $k+\frac{n}{2} <\beta \cdot C <k+\frac{n}{2}+1$, $\Real(z)>0$, and $\Real(z)+\frac{\Imag(z)^2}{H^2} > -\frac{\beta^2}{2}$. 

If $\beta \cdot C -\frac{n}{2}$ is an integer, then no such $k$ will exist. However, this problem can be avoided by simply scaling the class $\beta$. In fact, so long as $\beta \cdot C \neq 0$, then by replacing $\beta$ with $\left( 1-\frac{k+1}{\beta\cdot C} \right) \beta$, we can always choose $k$ to be $-1$. However, we will continue in more generality.

\begin{thm}[\cite{BG} \cite{BG2}]
	\label{lem:bg}
	For any Gieseker stable sheaf $E$ on $X$ which is torsion-free, $\ch_1(E)^2 \geq 2\ch_0(E) \ch_2(E)$.
\end{thm}

\begin{lem}
	\label{lem:stabfun}
	The function $Z_{H,\beta}$ is a stability function on $\CB_{H,k}^{-\Imag(z)}$, when $k$ is chosen so that $k+\frac{n}{2}<\beta \cdot C < k+\frac{n}{2}+1$ and $\Real{z}+\frac{\Imag{z}^2}{H^2}>-\frac{\beta^2}{2}$.
\end{lem}

\begin{proof}
	Any $E\cx \in \CB_{H,k}^{-\Imag(z)}$ fits into an exact triangle $$F[1] \rightarrow E\cx \rightarrow T\cx$$ for some $F \in \CF_{H,k}^{-\Imag(z)}$ and some $T\cx \in \CT_{H,k}^{-\Imag(z)}$. Since we have defined $Z_{H,\beta}$ using chern characters, which are additive on exact triangles, it follows that $$Z_{H,\beta}(E\cx)=Z_{H,\beta}(T\cx)-Z_{H,\beta}(F).$$ We have chosen $H$ so that $H \cdot C=0$ and so $\Imag(Z_{H,\beta}(E\cx))=\Imag(Z_{H,\beta}(T\cx))$. But $\Imag(Z_{H,\beta}(T\cx))=\Imag(Z_{H,\beta}(H^0(T\cx)))-\Imag(Z_{H,\beta}(H^{-1}(T\cx)))$. By the construction of $\CA_{H,k}^{-\Imag(z)}$, $\Imag(Z_{H,\beta}(H^0(T\cx))) >0$ and $\Imag(Z_{H,\beta}(H^{-1}(T\cx))) \leq 0$. 
	
	Now we must show that if $\Imag(Z_{H,\beta}(E\cx))=0$, then $\Real(Z_{H,\beta}(E\cx))<0$. 
	Consider the following diagram of short exact sequences in $\CB_{H,k}^{-\Imag(z)}$.
	\begin{center}
		\begin{tikzpicture}
		\matrix (m) [matrix of math nodes,column sep=2em, row sep=2em]
		{ & & H^{-1}(H^0_{\CA}(E\cx))[1] \\
			H^{-1}_{\CA}(E\cx)[1] & E\cx & H^0_{\CA}(E\cx) \\ 
			& & H^{0}(H^0_{\CA}(E\cx)) \\
		};
		
		\path[-stealth]
		(m-2-1) edge (m-2-2)
		(m-2-2) edge (m-2-3)
		(m-1-3) edge (m-2-3)
		(m-2-3) edge (m-3-3);
		\end{tikzpicture}
	\end{center}
	The equation $\Imag(Z_{H,\beta}(E\cx))=0$ holds if and only if the equations $$\Imag(Z_{H,\beta}(H^{-1}_{\CA}(E\cx)))=\Imag(Z_{H,\beta}(H^{-1}(H^0_{\CA}(E\cx)))=\Imag(Z_{H,\beta}(H^{0}(H^0_{\CA}(E\cx)))=0$$ also hold. Further, note that $H^{-1}_{\CA}(E\cx)) \in \CF_{H,k}^{-\Imag(z)}$, $H^{-1}(H^0_{\CA}(E\cx)) \in \CF_{H}^{-\Imag(z)}$, and $H^{0}(H^0_{\CA}(E\cx)) \in \CT_H^{-\Imag(z)}$. Thus we will be proceed by showing that for any sheaves $F \in \CF_{H,k}^{-\Imag(z)}$, $R \in \CT_H^{-\Imag(z)}$, and $S \in \CF_H^{-\Imag(z)}$ such that 
	$$\Imag(Z_{H,\beta}(F))=\Imag(Z_{H,\beta}(R))=\Imag(Z_{H,\beta}(S))=0,$$
	we have that $\Real(Z_{H,\beta}(R))<0$, $\Real(Z_{H,\beta}(F))>0$, and $\Real(Z_{H,\beta}(S))>0$. This will then show that $\Real(Z_{H,\beta})(E\cx)<0$. 
	
	First, suppose $F \in \CF_{H,k}^{-\Imag(z)}$. Note that for each complex $\CO_C(i)$ in $\CF_{H,k}^{-\Imag(z)}$, 
	$$Z_{H,\beta}(\CO_C(i))=-i-\frac{n}{2}+\beta\cdot C.$$ Since $i \leq k$, as long as $k$ is chosen so that $k< \beta \cdot C -\frac{n}{2}$, $\Real(Z_{H,\beta}(\CO_C(i)))>0$. Then since $Z_{H,\beta}$ is additive on exact triangles, $\Real(Z_{H,\beta}(F))>0$.
	
	Now let $R \in \CT_H^{-\Imag(z)}$ be such that $\Imag(Z_{H,\beta}(R))=0$. This implies that $\ch_0(R)=0$. Then $Z_{H,\beta}(R)=-\ch_2(R)+\beta\cdot \ch_1(R)$. Since $\ch_0(R)=0$, $R$ must be supported on either points or curves. If $R$ is supported at points, $\ch_2(R)$ will be positive and $\ch_1(R)=0$, so $Z(R)<0$. If $R$ is supported on a curve, it must be supported on $C$ since only $C \cdot H=0$. In particular, $R$ must be an extension of sheaves of the form $\CO_C(m)$ where $m > k$, since $R \in \CT_{H,k}^{-\Imag(z)}$. Since $Z_{H,\beta}(\CO_C(m))=-m-\frac{n}{2}+\beta \cdot C$, as long as $k$ is chosen so that $\beta \cdot C< k+1+\frac{n}{2}$. 
	
	Now let $S \in \CF_H^{-\Imag(z)}$ be such that $H \cdot \ch_1(S)+\Imag(z)\ch_0(S)=0$. In this case, $\ch_0(S)>0$, and so $\nu_H(S)=-{\rm Im}z$. Since $S$ is an object of $\CF_H^{-\Imag(z)}$ of maximal possible slope, $S$ is $\nu_H$-semistable. And so by Theorem \ref{lem:bg}, $\ch_1^2(S) \geq 2\ch_0(S)\ch_2(S)$. Then 
	\begin{align*}
	Z_{H,\beta}(S) &= \Real(z)\ch_0(S)+\beta \cdot \ch_1(S)-\ch_2(S) \\
	&\geq \ch_0(S)\left(\Real(z)+\frac{\beta \cdot \ch_1(S)}{\ch_0(S)}-\frac{\ch_1^2(S))}{2\ch_0^2(S)} \right) \\
	&=  \ch_0(S)\left(\Real(z)-\frac{(\ch_1(S)-\ch_0(S)\beta)^2}{2\ch_0^2(S)}+\frac{\beta^2}{2} \right). 
	\end{align*}
	
	Since $H\cdot (\ch_1(S)-\ch_0(S)\beta)=-\Imag(z) \ch_0(S)$, we can see that $$H\cdot \left(\ch_1(S)-\ch_0(S)\beta+\frac{\ch_0(F)\Imag(z)}{H\cdot H}H\right)=0.$$ Then by the Hodge Index Theorem, $$\left(\ch_1(S)-\ch_0(S)\beta+\frac{\ch_0(F)\Imag(z)}{H^2}H\right)^2 \leq 0$$ 
	We can now rewrite 
	\begin{align*}
	Z_{H,\beta}(S) &= \Real(z)\ch_0(S)+\beta \cdot \ch_1(S)-\ch_2(S) \\
	&\geq  \ch_0(S)\left(\Real(z)-\frac{(\ch_1(S)-\ch_0(S)\beta)^2}{2\ch_0^2(S)}+\frac{\beta^2}{2} \right) \\
	&=\ch_0(S)\left(\Real(z)-\frac{(\ch_1(S)-\ch_0(S)\beta+\frac{\ch_0\Imag(z)}{H^2}H)^2}{2\ch_0^2(S)}+\frac{\Imag(z)^2}{H^2}+\frac{\beta^2}{2} \right). 
	\end{align*}
	So as long as $\Real(z)+\frac{\Imag(z)^2}{H^2} > -\frac{\beta^2}{2}$, we have $Z_{H,\beta}(S)>0$. 
\end{proof}

\begin{lem}
	\label{lem:hn}
	The pair $(Z_{H,\beta},\CB_{H,k})$, with $H$ chosen to be a rational class, and $\Imag(z) \in \BQ$  satisfy the HN-property.
\end{lem}

\begin{proof}
	Following \cite[Proposition B.2]{BM2011}, we first show that the image of $\Imag(Z_{H,\beta}(\CB_{H,k}^{-\Imag(z)}))$ is discrete. This is clear, since the classes $\ch_1(E\cx)$ lie in a lattice for all $E\cx \in \CDb(X)$. Now for $E\cx \in \CB_{H,k}^{-\Imag(z)}$, we must show that for any sequence of inclusions $$0=A\cx_0 \hookrightarrow A\cx_1 \hookrightarrow \cdots \hookrightarrow A\cx_j \hookrightarrow A\cx_{j+ 1} \hookrightarrow \cdots \hookrightarrow E\cx$$ in $\CB_{H,k}^{-\Imag(z)}$, such that $\Imag(Z_{H,k}(A_j\cx))=0$ for all $j$, the sequence $A\cx_j$ stabilizes. 
	
	$E\cx$ lies in an exact triangle $$F[1] \rightarrow E\cx \rightarrow S\cx$$ with $F \in \CF_{H,\beta}^{-\Imag(z)}$ and $S\cx \in \CT_{H,\beta}^{-\Imag(z)}$. Suppose $S\cx$ has an HN filtration in $\CA_H^{-\Imag(z)}$. 
	That is, there exists an exact triangle 
	\begin{equation}
	\label{hn4}
	A\cx \rightarrow S\cx \rightarrow B\cx
	\end{equation}
	in $\CA_H$, such that $\Imag(Z_{H,\beta}(A\cx))=0$, and for all $C\cx \in \CA_H$ such that $\Imag(Z_{H,\beta}(C\cx))=0$, $\Hom(C\cx, B\cx)=0$. We can take the long exact sequence of cohomology of (\ref{hn4}) with respect to the heart $\CB_{H,k}$ to get an exact sequence $$H^0_{\CB}(A\cx) \rightarrow E\cx \rightarrow B\cx \rightarrow H^1_{\CB}(A\cx).$$ Let $D\cx$ be the cone of the morphism $H^0_{\CB}(A\cx) \rightarrow E\cx$. Then $D\cx$ is automatically in $\CT_{H,k}^{-\Imag(z)}$, and this is an exact triangle in $\CB_{H,k}^{-\Imag(z)}$. 
	
	Further, $\Imag(Z_{H,\beta}(H^0_{\CB}(A\cx)))=0$. Now suppose $C\cx$ lies in $\CB_{H,k}$ and $\Imag(Z_{H,\beta}(C\cx))=0$. Then $C\cx$ fits into an exact triangle $F\cx[1] \rightarrow C\cx \rightarrow T\cx$ with $F\cx \in \CF_{H,k}^{-\Imag(z)}$, $T\cx \in \CT_{H,k}^{-\Imag(z)}$, and ${\rm Im}Z_{H,\beta}(F\cx)=\Imag(Z_{H,\beta}(T\cx))=0$. Since $D\cx$ lies in $\CT_{H,k}^{-\Imag(z)}$ there can be no morphisms from $F\cx[1]$ to $D\cx$. There can be no morphisms $T\cx \rightarrow C\cx$ since such a morphism would imply that $\Hom(T\cx, D\cx) \neq 0$. Thus ${\rm Hom}(C\cx,D\cx)=0$. 
	
	Now consider the morphism $E\cx \rightarrow D\cx$. The kernel of this morphism $K\cx$ in $\CB_{H,k}^{-\Imag(z)}$ fits into an exact triangle $$F[1] \rightarrow K\cx \rightarrow A\cx.$$ Hence $K\cx \in \CB_{H,k}^{-\Imag(z)}$ and $\Imag(Z_{H,\beta}(K\cx))=0$. Therefore $E\cx$ also has the HN property in $\CB_{H,k}^{-\Imag(z)}$. Therefore, it is enough to show that if $E\cx \in \CT_{H,k}^{-\Imag(z)}$, then $E\cx$ has an HN filtration in $\CA_H^{-\Imag(z)}$. 
	
	We now prove that $E\cx$ has an HN-filtration in $\CA_H^{-\Imag(z)}$. This proof is similar to \cite[Proposition 7.1]{K3}, where we use the nef divisor $H$ instead of an ample divisor $\omega$. Suppose we have a sequence of inclusions  $$0=A\cx_0 \hookrightarrow A\cx_1 \hookrightarrow \cdots \hookrightarrow A\cx_j \hookrightarrow A\cx_{j+ 1} \hookrightarrow \cdots \hookrightarrow E\cx$$ in $\CA_H^{-\Imag(z)}$, where $\Imag(Z_{H,\beta}(A\cx_j))=0$ for all $j$. Then for each $j$ we have exact triangles 
	\begin{equation}
	\label{hn1}
	A\cx_{j-1} \rightarrow A\cx_j \rightarrow B\cx_j
	\end{equation}
	\begin{equation}
	\label{hn2}
	A\cx_j \rightarrow E\cx \rightarrow C\cx_j
	\end{equation}
	where $B\cx_j$ and $C\cx_j$ are in $\CA_H$. 
	
	Taking the long exact sequence of cohomology of (\ref{hn1}) and (\ref{hn2}) yields a sequence of inclusions in ${\rm Coh}(X)$:
	$$0=H^{-1}(A\cx_0) \hookrightarrow H^{-1}(A\cx_1) \hookrightarrow \cdots \hookrightarrow H^{-1}(A\cx_j) \hookrightarrow H^{-1}(A\cx_{j+1}) \hookrightarrow \cdots \hookrightarrow H^{-1}(E\cx).$$ Since $\Coh(X)$ is Noetherian, this sequence stabilizes. So we can assume that $H^{-1}(A\cx_j)$ is constant for all $j$. Then there is an exact sequence $$0 \rightarrow H^{-1}(B\cx_j) \rightarrow H^0(A\cx_{j-1})\rightarrow H^0(A\cx_j) \rightarrow H^0(B\cx_j) \rightarrow 0.$$ But $H^{-1}(B\cx_j)$ is torsion-free, and $H^0(A\cx_{j-1})$ is a torsion sheaf, so $H^{-1}(B\cx_j)=0$ for all $j$. 
	
	It remains to show that for $j \gg 0$, $H^0(B\cx_j)=0$. The triangles (\ref{hn1}) and (\ref{hn2}) yield a third triangle, 
	\begin{equation}
	\label{hn3}
	B\cx_{j} \rightarrow C\cx_{j-1} \rightarrow C\cx_j.
	\end{equation}
	The long exact sequence of cohomology of (\ref{hn2}) and (\ref{hn3}) together yield a sequence of surjections in ${\rm Coh}(X)$:
	$$H^0(E\cx) \twoheadrightarrow H^0(C\cx_1) \twoheadrightarrow \cdots \twoheadrightarrow H^0(C\cx_{j-1}) \twoheadrightarrow H^0(C\cx_j)\twoheadrightarrow \cdots.$$
	Since $\Coh(X)$ is Noetherian, this sequence stabilizes. So if we take $j \gg 0$, we can assume  $H^0(C_j)$ are constant. Then we have an exact sequence 
	\begin{equation}
	\label{hn6}
	0 \rightarrow H^{-1}(C\cx_{j-1}) \rightarrow H^{-1}(C\cx_j) \rightarrow H^0(B\cx_j) \rightarrow 0.
	\end{equation}
	
	Furthermore, from (\ref{hn2}) we see that for $j \gg 0$, the map $H^{-1}(A\cx) \rightarrow H^{-1}(E\cx)$ is constant. So there is a torsion-free sheaf $Q$ such that for all $j \gg 0$, $$0 \rightarrow Q \rightarrow H^{-1}(C_j\cx) \rightarrow H^0(A_j) \rightarrow  0$$ is exact. We would like to say that the sequence of inclusions
	\begin{equation}
	\label{hn5}
	0 \subseteq Q \subseteq H^{-1}(C_1\cx) \subseteq \cdots \subseteq H^{-1}(C_{j-1}\cx) \subseteq H^{-1}(C_j\cx) \subseteq \cdots
	\end{equation}
	stabilizes for $j \gg 0$
	
	If $H^{0}(A_j\cx)$ is supported on points for $j \gg 0$, then it follows from the argument of \cite[Proposition 7.1]{Bridgeland} that the sequence stabilizes for $j \gg 0$. Otherwise, $H^0(A_j\cx)$ is supported on $C$ for all $j$. Furthermore, since $H^0(A_j\cx) \in \CT_{H,k}^{-\Imag(z)}$, we can further assume that $\Hom(H^0(A_j\cx),\CO_C(k))=0$. Also,      (\ref{hn6}) implies that $H^0(B_j\cx)$ is the quotient $H^0(A_j\cx)/H^0(A_{j-1}\cx)$, and hence supported on points.  
	
\end{proof}

\begin{thm}
	\label{thm:stab}
	The pair $(Z_{H,\beta},\CB_{H,k}^{-\Imag(z)})$ define a stability condition on $\CDb(X)$ when $k$ is chosen so that $k+\frac{n}{2}<\beta \cdot C < k+\frac{n}{2}+1$ and $\Real(z)+\frac{\Imag(z)^2}{H^2} > -\frac{\beta^2}{2}$.
\end{thm}

\begin{proof}
	Lemmas \ref{lem:stabfun} and \ref{lem:hn} show that the pair $(Z_{H,\beta},\CB_{H,k}^{-\Imag(z)})$ satisfies the properties required in Definition \ref{def:stab}.
\end{proof}

\section{Support property}

In order to consider wall-crossing, we must show that when the pair $\sigma_{H,\beta}=(Z_{H,\beta},\CB_{H,k}^{-\Imag(z)})$ is deformed slightly, the phases of objects do not vary too much. That is, we need to show $\sigma_{H,\beta}$ satisfies the support property, stated in Definition \ref{defn:support}. This definition is equivalent to the following alternate definition, given in \cite[Section 2.1]{KS}.

\begin{prop}
	\label{prop:quad}
	A stability condition $\sigma=(Z,\CB)$ satisfies the support property if and only if there exists a quadratic form $Q$ such that $Q$ is negative definite on the kernel of the central charge $Z$, and for any $\sigma$-semistable objects $E\cx$ in $\CB$, $Q(E\cx) \geq 0$.
\end{prop}

The proof is given in \cite[Section 2.1]{KS} and in \cite[Appendix A]{BMS}. We will construct such a quadratic form for a range of stability conditions $\sigma_s$ we now define, by considering semistable objects in the limit as $s \rightarrow \infty$.

\begin{defn}
	\label{def:sstab}
	For every $s\geq  1$ we can define a new stability condition $\sigma_{H,\beta,s}=(Z_{H,\beta,s},\CB_{H,k}^{-\Imag(z)})$, where $\CB_{H,k}^{-\Imag(z)}$ is as before, and $$Z_{H,\beta,s}(E\cx)=-\ch_2(E\cx)+\beta\cdot \ch_1(E\cx)+s\Real(z)\ch_0(E\cx)+i(H\cdot \ch_1(E\cx)+\Imag(z)\ch_0(E\cx)).$$
\end{defn}

\begin{lem}
	The pair $\sigma_{H,\beta,s}=(Z_{H,\beta,s},\CB_{H,k}^{\Imag(z)})$ give a stability condition on $X$ when $\beta$ and $z$ satisfy the conditions of Lemma \ref{lem:stabfun} and $\Real(z)>0$.
\end{lem}

\begin{proof}
	We need to show that the image $Z_{H,\beta,s}(\CB_{H,k}^{\Imag(z)})$ lies in the upper half plane for $s \geq 1$. The case $s=1$ is shown in Lemma \ref{lem:stabfun}. When $s>1$, then $$s\Real(z)+\frac{\Imag(z)^2}{H^2}>\Real(z)+\frac{\Imag(z)^2}{H^2}>-\frac{\beta^2}{2},$$ and so the pair $\beta$,  $s\Real(x)+i\Imag(z)$ satisfy the conditions of Lemma \ref{lem:stabfun}, and $\sigma_{H,\beta,s}$ is also a stability condition on $X$.
\end{proof}

We will now describe what happens as $s$ grows large.

\begin{defn}
	Define $\calD$ to be the set of $E\cx$ in $\CB_{H,k}^{-\Imag(z)}$ such that $E\cx$ is $Z_{H,\beta,s}$-semistable for $s\gg 0$. 
\end{defn}

\begin{lem}
	\label{lem:larges} 
	If $E\cx$ is in $\calD$ then it is of one of the following forms:
	\begin{enumerate} 
		\item $E\cx$ is a slope semistable sheaf in $\CT_{H}^{-\Imag(z)}$.
		\item $H^0(E\cx)$ is either $0$ or supported on $C$ or on points, and $H^{-1}(E\cx)$ fits into an exact sequence 
		$$0\rightarrow G \rightarrow H^{-1}(E\cx) \rightarrow F \rightarrow 0$$ where $F$ is a slope semistable sheaf in $\CF_{H}^{-\Imag(z)}$, and $G \in \CF_{H,k}^{-\Imag(z)}$. Here $G$ must be 0 unless $\nu_H(G)=-\Imag(z)$.
	\end{enumerate}
\end{lem}

\begin{proof}
	Suppose that $E\cx$ is $Z_{H,\beta,s}$-semistable for $s \gg 0$. Recall that $E\cx$ fits into an exact triangle $$G[1] \rightarrow E \cx \rightarrow T\cx$$ where $G \in \CF_{H,k}^{-\Imag(z)}$, and that $T\cx$ must itself fit into an exact triangle $$F[1] \rightarrow T\cx \rightarrow S$$ where $F \in \CF_H^{-\Imag(z)}$ and $S \in \CT_H^{-\Imag(z)}$ are sheaves. Suppose first that $\ch_0(E\cx)>0$. Then as $s \rightarrow \infty$, $\phi_{H,\beta,s}(E\cx) \rightarrow 0$. Since $G[1]$ is fixed as $s$ varies with phase $1$, $G$ must be $0$. Further, since $F$ is a sheaf, $\ch_0(F[1])<0$. So as $s \rightarrow \infty$, $\phi_{H,\beta,s}(F[1]) \rightarrow 1$.

	Note that since $\Ext^{-1}(F,\CO_C(l))=0$ for all values of $l$, $F[1] \in \CT_{H,k}^{-\Imag(z)}$. Since $S \in \CT^{-\Imag(z)}_{H,k}$ as well, $F[1]$ is a subobject of $T\cx$ in $\CB_{H,k}^{-\Imag(z)}$. But we've assumed that $\phi_{H,\beta,s}(E\cx)\rightarrow 0$, and so $F[1]=0$ as well.
	
	Now $E\cx \cong S$ is a sheaf in $\CT_H^{-\Imag(z)}$ with $\ch_0(S)>0$. We can write the HN-filtration of $S$ with respect to $\nu_H$:
	$$0 \rightarrow S_1 \hookrightarrow \cdots \hookrightarrow S_{m-1}\hookrightarrow S_m=S $$ with quotients $T_i := S_i/S_{i-1}$ which are $\nu_H$-semistable, and with $\nu_H(T_i)>\nu_H(T_{i+1})$ for $i=1, \dots, m-1$. It may be that $S_{m-1}$ is not itself in $\CB_{H,k}^{-\Imag(z)}$, but it is in $\CA_H$, and so there is an exact triangle $$S'_{m-1} \rightarrow S_{m-1} \rightarrow F' $$ with $S'_{m-1} \in \CT_{H,k}^{-\Imag(z)}$ and $F' \in \CF_{H,k}^{-\Imag(z)}$. Since $F'$ is supported on $C$, $\nu_H(S_{m-1}')=\nu_H(S_{m-1})$. 
	
	Checking the long exact cohomology sequence, we can see that $S'_{m-1}$ is a sheaf, and so we can compose maps to get an injective map of sheaves $S'_{m-1} \hookrightarrow S$. The quotient will be a sheaf of positive slope, since $S \in \CT_{H}^{-\Imag(z)}$. Further, it can have no maps to $\CF_{H,k}^{-\Imag(z)}$, otherwise this would contradict that $S$ has no such maps. And so $S'_{m-1}$ injects into $S$ in $\CB_{H,k}^{-\Imag(z)}$. Since $\nu_H(S'_{m-1})>\nu_H(S)$, for $s$ sufficiently large, $\phi_{H,\beta,s}(S'_{m-1})>\phi_{H,\beta,s}(S)$, contradicting that $S$ is stable. Thus $S$ must itself be slope-semistable.
	
	Now suppose $\ch_0(E\cx)<0$. Then as $s \rightarrow \infty$, $\phi_{sH,\beta}(E\cx) \rightarrow 1$. Since $S$ is a quotient of $E\cx$ in $\CB_{H,k}^{-\Imag(z)}$, it must also be that $\phi_{H,\beta,s}(S) \rightarrow 1$ as $s \rightarrow \infty$. This is possible only if $\ch_0(S)=0$ and $H \cdot \ch_1(S)=0$. And so $S$ must be supported at points or along $C$. If $H \cdot \ch_1(F)<\Imag(z)$ then $G$ must be $0$.  In this case, we can use HN-filtrations in the same manner as in the previous case to show that $F$ must be slope semistable itself. If $H \cdot \ch_1(F)=-\Imag(z)$, then $F$ is automatically slope semistable. 
	
	It remains to consider $\ch_0(E\cx)=0$. In this case, $\ch_0(T\cx)=0$. It is possible that $T\cx=0$, since $\phi_{H,\beta,s}(G[1])=1$ for any value of $s$. If $T\cx \neq 0$, then first suppose $H \cdot \ch_1(E\cx) >-\Imag(z)$. This implies that as $s \rightarrow \infty$, $\phi_{H,\beta,s}(E\cx) \rightarrow \frac{1}{2}$. And so $G=0$, and $F$ must be $0$ as well. Then $E\cx$ is a torsion sheaf supported on a curve $C'$ not contained in $C$. If $H \cdot \ch_1(E\cx)=-\Imag(z)$, then $F$ must again be $0$, and now $S$ must be a torsion sheaf supported on $C$ or on points. 
\end{proof}

We now work towards the construction of a quadratic form $Q$ which will satisfy the requirements of Proposition \ref{prop:quad} where the semistable objects are the objects of $\calD$. We first need the following lemmas.

\begin{lem}
	\label{lem:otherc}
	There is a positive constant $C_H$ depending only on $H$ so that for any sheaf $E$ supported on a curve $C' \not\subseteq C$,
	$$H^2\ch_1(E)^2 + C_H(H\cdot \ch_1(E))^2 \geq 0.$$
\end{lem}

\begin{proof}
	Write $C'=\alpha+lC$ with $\alpha$ a class in $C^{\perp}$. Since $C'$ is not contained in $C$, $C \cdot C' \geq 0$, so $l \leq 0$. Further, for $0<t \ll 1$, $H-tC$ is ample. This follows from the fact that $H$ is big and nef, and that $C$ is the only effective divisor in $H^{\perp}$.
	
	Since this is an ample divisor, $C'\cdot (H-tC)>0$. So $H\cdot \alpha>-tln\geq 0$. Then 
	\begin{align*}
	(C')^2 +\frac{1}{t^2n}(H\cdot \alpha )^2 &= \alpha^2-l^2n+\frac{1}{t^2n}(H\cdot \alpha )^2 \\
	&> \alpha^2. \\
	\end{align*}
	
	Further, since $H$ is nef, the Hodge Index Theorem states that there exists some constant $C_H>0$ depending only on $H$ so that $H^2\alpha^2+C_H(H\cdot \alpha)^2 \geq 0$.
	We then have
	
	\begin{align*}
	(C')^2+\left(\frac{1}{t^2n}+\frac{C_{H}}{H^2} \right)(H\cdot \alpha)^2 &= \alpha^2-l^2n+\left(\frac{1}{t^2n}+\frac{C_{H}}{H^2} \right)(H\cdot \alpha)^2 \\
	&> \alpha^2+\frac{C_{H}}{H^2}(H\cdot \alpha)^2 \\
	&\geq 0. \\
	\end{align*}
\end{proof}

Define a constant $D_H$ as follows, where $C_H$ is as in Lemma \ref{lem:otherc}: $$m_1={\rm max}\{H\cdot \ch_1(F) \ \vert \ F \ {\rm is \ a \ slope \ semistable \ sheaf}, \ H\cdot \ch_1(F)<-{\rm Im z}, \, \ch_0(F)=1\}.$$ $$m_2={\rm max}\{H\cdot \ch_1(F) \ \vert \ F \ {\rm is \ a \ slope \ semistable \ sheaf}, \ H \cdot \ch_1{F}< -{\rm Im}z, \ \ch_0(F)=2\}.$$
$$D_H={\rm max}\left\{\frac{\frac{3}{2}n+2k+3}{m_1^2},\frac{8k+21}{m_2^2},C_H\right\}.$$ We now define a preliminary quadratic form.

\begin{defn}
	\label{def:q0}
	$$Q_0(E\cx):=\ch_1(E\cx)^2-2\ch_0(E\cx)\ch_2(E\cx)+D_H({\rm Im}Z_{H,\beta}(E\cx))^2.$$
\end{defn}

\begin{lem}
	\label{lem:q0}
	$Q_0(E\cx)\geq 0$ for $E\cx$ in $\calD$ such that $\Imag(Z_{H,\beta}(E\cx))>0$.
\end{lem}

\begin{proof}
	First, if $E\cx$ is a torsion-free sheaf or a shift of a torsion-free sheaf in $\calD$, then by Lemma \ref{lem:larges}, this sheaf is $\nu_H$-semistable. Thus $Q_0(E\cx) \geq 0$ by Theorem \ref{lem:bg}. If $E\cx$ is a torsion sheaf not supported on $C$, then it is either supported on points, in which case $Q_0(E\cx)=0$, or it is supported on a curve not contained in $C$. In this case, $Q_0(E\cx) \geq 0$ by Lemma \ref{lem:otherc}.
	
	It remains to consider $E\cx$ such that there is an exact triangle $$F[1] \rightarrow E\cx \rightarrow T$$ where $T$ is a torsion sheaf supported on $C$ or on points, and $F$ is a slope semistable sheaf of slope smaller than $0$. If $\nu_H(F)< \Imag(z)$, then $\Hom(\CO_C(k+1),F[1])=0$ since both are semistable, and $\phi_{H,\beta}(F[1])<1$. Further, $\Ext^2(\CO_C(k+1),F[1])=\Ext^3(\CO_C(k+1),F)=0$. Thus $\chi(\CO_C(k+1),F[1])\leq 0$. By Hirzebruch-Riemann-Roch, $\chi(\CO_C(k+1),F[1])=-C\cdot \ch_1(F[1])+(k+2)\ch_0(F[1])$. Combining these facts, $C\cdot \ch_1(F[1]) \geq (k+2)\ch_0(F[1])$.
	
	Now we have
	\begin{align*}
	Q_0(E\cx) &\geq (\ch_1(T)+\ch_1(F[1]))^2-2\ch_0(F[1])(\ch_2(T)+\ch_2(F[1]) \\
	&\geq \ch_1(T)^2+2\ch_1(F[1])\ch_1(T)-2\ch_0(F[1])\ch_2(T). \\
	\end{align*}
	If $T$ is supported on points this is clearly positive. It suffices to consider $T \cong \CO_C(l)$ where $l>k$. Then the above inequality becomes
	\begin{align*}
	Q_0(E\cx) &\geq -n+2C\cdot \ch_1(F[1])-2\ch_0(F[1])\left(l+\frac{n}{2}\right) \\
	&\geq -n+2\left(k+2-l-\frac{n}{2}\right)\ch_0(F[1]). \\
	&= (\ch_0(F)-1)n+2\ch_0(F)(l-k-2) \\
	\end{align*}
	For $\ch_0(F) \geq 3$ this is necessarily positive. The only cases in which it may not be positive are $l=k+1$ and $\ch_0(F[1])=-1$, or $l=k+1$, $\ch_0(F[1])=-2$ and $n=3$. In these cases, the choice of $D_H$ ensures that $Q_0$ is positive.
\end{proof}

\begin{lem}
	\label{lem:q0neg}
	$Q_0$ is negative definite on the kernel of $Z_{H,\beta,s}$ as defined in Definition \ref{def:sstab} for all $s\geq1$. 
\end{lem}

\begin{proof}
	Suppose $Z_{H,\beta,s}(E\cx)=0$ for some $s \geq 1$.  Note that if $\ch^{\beta}(E\cx)=\ch(E\cx) e^{-\beta}$, $(\ch_1^{\beta})^2(E\cx)-2\ch_0^{\beta}(E\cx)\ch_2^{\beta}(E\cx)=\ch_1^2(E\cx)-2\ch_0(E\cx)\ch_2(E\cx)$. Since $\beta \cdot H=0$, $\ch_1^{\beta}(E\cx)=\ch_1(E\cx)-\beta\ch_0(E\cx)$ is in $H^{\perp}$, so $(\ch_1^{\beta})^2(E\cx) \leq 0$ by the Hodge Index Theorem. Further, since $E\cx$ is in the kernel of $Z_{H,\beta,s}$, $\ch_2^{\beta}(E\cx)=(s^2z+\frac{\beta^2}{2})\ch_0(E\cx)$ has the same sign as $\ch_0^{\beta}(E\cx)=\ch_0(E\cx)$. And so $Q_0(E\cx) \leq 0$. 
\end{proof}

$Q_0$ is negative on sheaves supported on $C$, and on their shifts. We now must adjust $Q_0$ to find a quadratic form which is positive on such sheaves. Note that it suffices to consider sheaves $\CO_C(l)$ where $l >k$, and shifts $\CO_C(m)[1]$, where $m\leq k$. 

\begin{defn}
	\label{defn:qs}
	Let $m_{\beta}={\rm min}\{\vert \beta \cdot C-k-\frac{n}{2}-1 \vert, \vert k+\frac{n}{2}-\beta\cdot C \vert \}$ and $D_{\beta}=\frac{n}{m_{\beta}^2}$. We now define another preliminary set of quadratic forms for $s \geq 1$: $$Q_s(E\cx)=Q_0(E\cx)+D_{\beta}(\Real(Z_{sH,\beta}(E\cx)))^2.$$
\end{defn}
By construction, $Q_s(E\cx) \geq 0$ for all $E\cx \in \calD$, and $Q_s$ is negative definite on the kernel of $Z_{H,\beta,s}$.

\begin{thm}
	\label{thm:supportbr}
	The central charge $Z_{H,\beta}$ satisfies the support property in the sense of Proposition \ref{prop:quad} for Bridgeland semistable objects in $\CB_{H,k}^{-\Imag(z)}$ with respect to the quadratic form $Q_1$. 
\end{thm}

\begin{proof}
	First we consider $E\cx \in \CB_{H,k}^{-\Imag(z)}$ such that $\Imag(Z_{H,\beta}(E\cx))>0$. The imagine of $\Imag(Z_{H,\beta})$ is discrete, and so we may proceed by induction. Any objects for which $\Imag(Z_{H,\beta})$ is minimal must be in $\calD$, as any possible destabilizing subobjects must have smaller imaginary part. Lemma \ref{lem:q0} and Lemma \ref{lem:q0neg} show that the support property is satisfied for such objects. 
	
	Now suppose there is some $E\cx \in \CB_{H,k}^{-\Imag(z)}$ which is $Z_{H,\beta}$-semistable but for which $Q_0(E\cx)<0$. Assume that for any $F\cx$ such that $\Imag(Z_{H,\beta}(F\cx))<\Imag(Z_{H,\beta}(E\cx))$, the requirements of the support property are met by $Q_0$. Since $E\cx$ is not in $\calD$, this implies that there exists some $s>1$ for which $E\cx$ is strictly $Z_{H,\beta,s}$-semistable. Let $E\cx_1,\dots , E\cx_m$ be the Jordan-H\"{o}lder factors of $E\cx$. Then $\Imag(Z_{H,\beta}(E\cx_i))<\Imag(Z_{H,\beta}(E\cx))$ for all $i=1,\dots, m$. And so by the inductive hypothesis, $Q_0(E_i)\geq 0$.
	
	The quadratic form $Q_0$ divides $K(\CDb(X))$ into a positive and negative cone. For any pair $E\cx_i$ and $E\cx_j$ of Jordan H\"{o}lder factors of $E\cx$, these lie on the same ray in the image of $Z_{H,\beta,s}$. And so there is some $a>0$ for which $Z_{H,\beta,s}(E\cx_i)-aZ_{H,\beta,s}(E\cx_j)=0$. The restriction of $Q_0$ to the kernel of $Z_{H,\beta,s}$ is negative definite, and so this combination $[E\cx_i]-a[E\cx_j]$ must lie in the negative cone of $Q_0$. This implies that any positive linear combination of $[E\cx_i]$ and $[E\cx_j]$ lies in the positive cone of $Q_0$. Since this is true for any pair $E\cx_i$ and $E\cx_j$, it follows that $Q_0(E\cx) \geq 0$. 
	
	We have shown that $Q_0$ satisfies the requirements of the support property for semistable objects of strictly positive imaginary part. We can now use $Q_1$ from Definition \ref{defn:qs} which will satisfy the support property for all $Z_{H,\beta}$-semistable objects. 
\end{proof}

The above shows that $(Z_{H,\beta},\CB_{H,k}^{-\Imag(z)})$ is a stability condition with the support property when $H$ and $\Imag(z)$ are rational. We need to extend this results to real $H$ and $\Imag(z)$. 

\begin{thm}
	\label{thm:real}
	The pair $(Z_{H,\beta},\CB_{H,k}^{-\Imag(z)})$ is a stability condition with the support property for $H$ and $\Imag(z)$ real. 
\end{thm}
\begin{proof}
	By Theorem \ref{thm:supportbr} this holds for $H$ and $-\Imag(z)$ in $\BQ$. Then we can deform these stability conditions to have stability conditions on $\BR$. It remains to show that this is well-defined. This holds by an argument similar to those in  \cite[Section 5]{Toda14} and  \cite[Appendix B]{BMS}.
	
	For each stability condition $\sigma_{H,\beta}=(Z_{H,\beta},\CB_{H,k}^{-\Imag(z)})$ with $\Imag(z)$ and $H$ rational, we can obtain an open subset of the space of stability conditions by deforming $\sigma_{H,\beta}$ \cite[Proposition A.5]{BMS}. This gives a cover of the wall of the geometric chamber.
	
	If $\sigma_{H_1,\beta_1}$ and $\sigma_{H_2,\beta_2}$ are two such stability conditions, and $U_1$ and $U_2$ are the corresponding open subsets, it remains to show that deforming $\sigma_{H_1,\beta_1}$ in $U_1$ and $\sigma_{H_2,\beta_2}$ in $U_2$ gives the same stability conditions in $U_1 \cap U_2$. It would suffice to show that there exists a stability condition $\sigma_{H,\beta} \in U_1 \cap U_2$ where this holds. But $U_1 \cap U_2$ contains stability conditions in the geometric chamber of ${\rm Stab}(X)$. Since this holds inside the geometric chamber, it thus holds on the wall.
\end{proof}

\section{Wall-crossing}

\label{sec:moduli}

We now consider a stability condition $\tau$ across the wall constructed in the previous section. We will construct a moduli space $M_{\tau}([\CO_x])$ for stable objects of class $[\CO_x]$. First we determine the $\tau$-stable objects of this class by deforming $\sigma_{H,\beta}$.

\begin{lem}
	\label{lem:jh}
	$\CO_C(k+1)$ and $\CO_C(k)[1]$ are simple objects in $\CB_{H,k}^{-\Imag(z)}$.
\end{lem}

\begin{proof}
	Suppose $A\cx$ is a subobject of $\CO_C(k)[1]$ in $\CB_{H,k}^{-\Imag(z)}$. Then there is an exact triangle $A\cx \rightarrow \CO_C(k)[1] \rightarrow B\cx$ for some $B\cx \in \CB_{H,k}^{-\Imag(z)}$. Since we assume $A\cx$ and $B\cx$ are in the heart $\CB_{H,k}^{-\Imag(z)}$ we know that $A\cx$ and $B\cx$ have cohomology only in degrees $-1$ and $0$. Hence we get the following long exact sequence by taking cohomology in $\Coh(X)$. $$0 \rightarrow H^{-1}(A\cx) \rightarrow \CO_C(k) \rightarrow H^{-1}(B\cx) \rightarrow H^0(A\cx) \rightarrow 0.$$ 
	Further, since $B\cx$ is a quotient of $\CO_C(k)[1]$, which lies in $\CF_{H,k}^{-\Imag(z)}$, we have $B\cx \in \CF_{H,k}^{-\Imag(z)}$.
	
	We see from the sequence above that $H^{-1}(A\cx)$ is a sheaf which injects into $\CO_C(k)$. This leaves only a few possiblities for which sheaf $H^{-1}(A\cx)$ can be.
	If $H^{-1}(A\cx)$ is a proper subsheaf of $\CO_C(k)$, then $H^{-1}(A\cx) \cong \CO_C(l)$ for some $l<k$. The quotient $H^{-1}(A\cx) \rightarrow \CO_C(k)$ is then supported on points. But such a quotient could not inject into $H^{-1}(B\cx)$, since all sheaves supported on points lie in $\CT_{H,k}^{-\Imag(z)}$.
	
	This leaves only the possibility that $H^{-1}(A\cx)$ is not a proper subsheaf of $\CO_C(k)$. That is, $H^{-1}(A\cx)$ is $0$ or $\CO_C(k)$. If $H^{-1}(A\cx) \cong \CO_C(k)$, then $H^{-1}(B\cx) \cong H^0(A\cx)$. This implies that these sheaves are both $0$, and $A\cx \cong \CO_C(k)[1]$. If $H^{-1}(A\cx) \cong 0$ then since $H^{-1}(B\cx)$ lies in $\CF_{H,k}^{-\Imag(z)}$, $H^0(A\cx)\cong 0$ and $B\cx \cong \CO_C(k)[1]$. 
	
	Now suppose $A\cx$ is a subobject of $\CO_C(k+1)$ and fits into an exact triangle $A\cx \rightarrow \CO_C(k+1) \rightarrow B\cx$. Again, taking cohomology with respect to $\CA_{H}^{-\Imag(z)}$ and ${\rm Coh}(X)$ separately, we can deduce that $A\cx$ is a sheaf supported on $C$, and that there is an exact sequence $$0 \rightarrow H^{-1}(B\cx) \rightarrow A\cx \rightarrow \CO_C(k+1) \rightarrow H^0(B\cx)\rightarrow 0.$$  Further $H^{-1}(B\cx) \in \CF_{H,k}^{-\Imag(z)}$, and $H^0(B\cx)$ is supported on $C$ or points. 
	
	If $H^0(B\cx)$ were supported on points, then the kernel of the map $\CO_C(k+1) \rightarrow H^0(B\cx)$ would be a sheaf in $\CF_{H,k}^{-\Imag(z)}$, from which $H^0(A\cx)$ could have no morphisms. And so $H^0(B\cx)$ can be only $\CO_C(k+1)$ or $0$. In the first case, $A\cx\cong 0$ and $B\cx\cong \CO_C(k+1)$. In the second case, $A\cx\cong \CO_x$ and $B\cx\cong 0$. 
\end{proof}

\begin{lem}
	\label{prop:otherx}
	If $x \in X \setminus C$, then $\CO_x$ is $\sigma_{H,\beta}$-stable. If $x \in C$, $\CO_x$ is strictly $\sigma_{H,\beta}$-semistable, destabilized by the exact triangle $$\CO_C(k+1)\rightarrow \CO_x \rightarrow \CO_C(k)[1].$$
\end{lem}

\begin{proof}
	Since $\CO_x$ is stable inside the geometric chamber, it is either $\sigma_{H,\beta}$-stable or it is $\sigma_{H,\beta}$-semistable. Suppose it is semistable. Then there is an exact triangle $$A\cx \rightarrow \CO_x \rightarrow B\cx$$ in $\CB_{H,k}^{-\Imag(z)}$ destabilizing $\CO_x$. Taking cohomology, we see that $A\cx$ is a sheaf, and that $H^0(B\cx)$ is either $0$ or $\CO_x$. In the latter case, $H^{-1}(B\cx) \cong H^0(A\cx) \cong 0$, so $B\cx \cong \CO_x$. 
	
	In the first case, we see $A\cx$ must be a torsion sheaf supported on $C$ or points, and $B\cx \cong F[1]$ for some sheaf $F \in \CF_{H,k}^{-\Imag(z)}$. Such a sequence can only exist when $x \in C$, so otherwise $\CO_x$ is stable. For points $x$ on $C$, the sequence $\CO_C(k+1) \rightarrow \CO_x \rightarrow \CO_C(k)[1]$ destabilizes $\CO_x$. 
\end{proof}

\begin{lem}
	\label{lem:jhtau}
	Suppose $E\cx$ is of class $[\CO_x]$ and is $\sigma_{H,\beta}$-semistable, then the only possible  Jordan-H\"older factors of $E\cx$ are $\CO_C(k+1)$ and $\CO_C(k)[1]$, or $\CO_x$ for some $x \not\in C$.
\end{lem}
\begin{proof}
	The Jordan-H\"older factors of $\CO_x$ must lie in the saturation of the lattice generated by $\CO_C$ and $\CO_x$ in $\CB_{H,k}^{-\Imag(z)}$. The objects $\CO_C(k+1)$, $\CO_C(k)[1]$, and $\CO_x$ are simple objects in this lattice. Suppose there is another simple object $E\cx$ in this lattice. Note that $\ch_0(E\cx)=0$, and $H\cdot \ch_1(E\cx)=0$. 
	
	We know $E\cx$ fits into an exact triangle $F[1] \rightarrow E\cx \rightarrow T\cx$ where $F \in \CF_{H,k}^{-\Imag(z)}$ and $T\cx \in \CT_{H,k}^{-\Imag(z)}$. So one of $F$ and $T\cx$ must be $0$. 
	If $E\cx=F[1]$, and is simple, we claim $E\cx \cong \CO_C(k)[1]$. To see this, note that since $\CF_{H,k}^{-\Imag(z)}$ was constructed as the extension closure of the set of objects of the form $\CO_C(l)$ for some $l \leq k$, all objects in $\CF_{H,k}^{-\Imag(z)}[1]$ have a morphism to $\CO_C(l)[1]$ for some $l\leq k$ which is surjective in $\CB_{H,k}^{-\Imag(z)}$. Since $E\cx=0$, we see $E\cx \cong \CO_C(l)[1]$ for this $l$. But then if $l \neq k$, there is an exact triangle in $\CB_{H,k}^{-\Imag(z)}$, $$T \rightarrow \CO_C(l)[1] \rightarrow \CO_C(k)[1]$$ where $T$ is a sheaf supported on points of length $k-l$. Hence $k=l$.
	
	If $E\cx=T\cx$, then since $H \cdot \ch_1(E\cx)=\ch_0(E\cx)=0$, $E\cx$ must be a sheaf supported on $C$ or on points. If $E\cx$ is supported on points and simple, then $E\cx$ is a skyscraper sheaf $\CO_x$ where $x \not\in C$. If $E\cx$ is supported on $C$, and $E\cx \in \CT_{H,k}^{-\Imag(z)}$, then $E\cx$ has $\CO_C(k+1)$ as a subobject. Hence since $E\cx$ is simple, $E\cx \cong \CO_C(k+1)$. 
\end{proof}

We will study the moduli space of $\tau$-stable objects of class $[\CO_x]$, where $\tau$ is a stability condition across the wall along which $\sigma_{H,\beta}$ lies. In order to study objects of this class, we will look at a local model and study a neighbourhood of the curve $C$ in $X$. Let $\CDb_C(X)$ denote the subcategory of $\CDb(X)$ of objects supported on $C$. Let $\widehat{X}$ be the completion of $X$ at $C$.

\begin{lem}
	\label{lem:local}
	$\CDb_C(X) \cong \CDb_C(\widehat{X})$.
\end{lem}
\begin{proof}
	By Proposition 1.7.11 in \cite{KaSch}, $\CDb_C(X) \cong \CDb(\Coh_C(X))$ and $\CDb_C(\widehat{X}) \cong \CDb(\Coh_C(\widehat{X}))$. It remains to show that $\Coh_C(X) \cong \Coh_C(\widehat{X})$. Any sheaf $\CF \in \Coh_C(X)$ is supported in a finite-order neighbourhood $C_n$ of $C$ in $X$. The embedding $\Coh(C_n) \rightarrow \Coh_C(X)$ is fully faithful. Similarly for $\widehat{X}$, any sheaf in $\Coh_C(\widehat{X})$ is supported on a finite-order neighbourhood of $C$, isomorphic to $C_n$ by construction. Since $\Coh(C_n) \rightarrow \Coh_C(\widehat{X})$ is also a fully faithful embedding, it follows that $\Coh_C(X) \cong \Coh_C(\widehat{X})$. 
\end{proof}

\begin{lem}
	\label{lem:local2}
	$\widehat{X}$ is isomorphic to the completion of ${\rm Tot} \ \CO_{\BP^1}(-n)$ at the $0$-section.
\end{lem}
\begin{proof}
	The curve $C$ is contractible. Up to isomorphism, there is a unique local singularity to which $\widehat{X}$ contracts. Further, the completion of ${\rm Tot} \ \CO_{\BP^1}(-n)$ at the $0$-section is another $-n$-curve, and hence it must contract to the same singularity. This local singularity has a unique minimal resolution, and so $\widehat{X}$ and the completion of ${\rm Tot} \ \CO_{\BP^1}(-n)$ at the $0$-section must be isomorphic.
\end{proof}

We will now construct a family of $\tau$-semistable objects of class $[\CO_x]$ in $K_0(X)$, with the goal of constructing a universal family over $M_{\tau}([\CO_x])$. We will do this by considering stable objects of the form $\CO_x$ for some $x \in X\setminus C$ and stable objects of the form $\eta(y)$ for some $y \in \BP\Ext^1(\CO_C(k+1),\CO_C(k)[1])$ separately, and then gluing along $C$.

Inside the geometric chamber of ${\rm Stab}(X)$, the stable objects of class $[\CO_x]$ are the skyscraper sheaves $\CO_x$ themselves. Hence a family is given by the object $\CO_{\Delta_X}$ in $\CDb(X \times X)$. However, along the wall we have constructed we will construct a new family of $\tau$-stable objects via semistable reduction.

Consider the following diagram.
\begin{center}
	\begin{tikzpicture}
	\matrix (m) [matrix of math nodes,row sep=3em,column sep=4em,minimum width=2em]
	{
		C \times X & X \times X \\
		C & X \\};
	\path[-stealth]
	(m-1-1) edge  (m-2-1)
	edge  node [above] {$j$} (m-1-2)
	(m-2-1.east|-m-2-2) edge
	(m-2-2)
	(m-1-2) edge(m-2-2);
	\end{tikzpicture}
\end{center}

There is an exact triangle in $\CDb(C\times X)$ as follows:
$$\CO_{C\times C}(-k-2,k) \rightarrow \CO_{C \times C}(-k-1,k+1) \rightarrow \CO_{\Delta_C}.$$  
Define $\CE$ by 
\begin{align}
\label{defofe}
\CE \rightarrow \CO_{\Delta_X} \rightarrow j_*\CO_{C\times C}(-k-2,k)[1]
\end{align} where the second map is given by the composition of the map coming from the exact triangle and the restriction map $\CO_{\Delta_X} \rightarrow j_* \CO_{\Delta_C}$.

First, note that $\CE$ is a sheaf. It fits into the exact sequence $$0 \rightarrow j_*\CO_{C \times C}(-k-2,k) \rightarrow \CE \rightarrow \CO_{\Delta_X}\rightarrow 0$$ on $X \times X$. $\CE$ is supported on $(C\times C) \cup_{\Delta_C} \Delta_X$. Using the octahedral axiom, we can say further that $\CE$ fits into the exact sequence $$0 \rightarrow \CO_{\Delta_X}(-C) \rightarrow \CE \rightarrow j_*\CO_{C \times C}(-k-1,k+1) \rightarrow 0.$$ We can see from this that $\CE \cong \CO_S$, where $S$ is the surface $(C\times C) \cup_{\Delta_C} \Delta_X$. 

For any point $x \in X$ there is an inclusion map $j_x \colon x \times X \hookrightarrow X \times X$. If we consider the pullback of \ref{defofe} via $j_x$, we obtain the exact triangle
$$\BL j_x^*\CO_{C\times C}(-k-2,k) \rightarrow \BL j_x^* \CO_S \rightarrow \CO_{x \times x}.$$
If $x \in X \setminus C$, $\BL j_x^* \CO_{C \times C}(-k-2,k) \cong 0$. This shows that $\BL j_x^* \CO_S \cong \CO_{x \times x}$, the skyscraper sheaf of the point $x \times x \in X \times X$. On the other hand, if $x \in C$, $\BL j_x^*\CO_{C \times C}(-k-2,k)\cong \CO_{x \times C}(k)[1]\oplus \CO_{x \times C}(k)$. Hence $\BL j_x^* \CO_S$ fits into an exact sequence $$\CO_{x \times C}(k)[1] \rightarrow \BL j_x^* \CO_S \rightarrow \CO_{x \times x} \rightarrow \CO_{x \times C}(k)[1].$$ The kernel of the map $\CO_{x \times x} \rightarrow \CO_{x \times C}(k)[1]$ is $\CO_{x \times C}(k+1)$, and so this shows that $\BL j_x^* \CO_S$ is isomorphic to a class in $\BP \Ext^1(\CO_C(k+1),\CO_C(k)[1]$, if this class in nontrivial. That is, $\CO_S$ defines a family of $\tau$-stable objects of class $[\CO_x]$.

\begin{lem}
	\label{lem:exts}
	There is a $\BP^{n-1}$ parametrizing $\tau$-stable objects of class $[\CO_x]$ which are not isomorphic to $\CO_x$ for any $x \in X$. 
\end{lem}

\begin{proof}
	
	Suppose $E\cx$ is a $\tau$-stable object of class $\CO_x$ for some $x \in C$. Then $E\cx$ must be strictly $\sigma_{H,\beta}$-semistable. By Lemma \ref{lem:jhtau} the Jordan-H\"older factors of $E\cx$ must be $[\CO_C(k+1)]$ and $[\CO_C(k)[1]]$. 
	
	We will now work in the local model described in Lemma \ref{lem:local2}. Since the sheaves $\CO_x$ were destabilized by the triangle $$\CO_C(k+1) \rightarrow \CO_x \rightarrow \CO_C(k)[1]$$ we know that $\phi_{\tau}(\CO_C(k+1))>\phi_{\tau}(\CO_C(k)[1])$. Hence since $E\cx$ is $\tau$-stable, it must fit into an exact triangle $$\CO_C(k)[1]\rightarrow E\cx \rightarrow \CO_C(k+1).$$ This means that the new $\tau$-stable objects $E\cx$ of class $[\CO_x]$ are parametrized by $\BP\Ext^1(\CO_C(k+1),\CO_C(k)[1])$.
	
	We can calculate the dimension ${\rm Ext}^1(\CO_C(k+1),\CO_C(k)[1])$ as the dimension of $H^2(X,\CO_C(k) \otimes \CO_C(k+1)^{\vee})$. The sheaf $\CO_C(k+1)$ is quasiisomorphic to the complex $\CO_X(-C)(k+1) \rightarrow \CO_X(k+1)$. Then $\CO_C(k+1)^{\vee}$ is quasiisomorphic to the complex $\CO_X(-k-1) \rightarrow \CO_X(C)(-k-1)$. Tensoring with $\CO_C(k)$, we now want to calculate $H^2(X,\CO_C(-1) \rightarrow \CO_C(-n-1))$. Note that $n>0$, and so there are no morphisms from $\CO_C(-1)$ to $\CO_C(-n-1)$. Hence we must compute $H^2(X,\CO_C(-1)\oplus \CO_C(-n-1)[-1])$ This is the direct sum $H^2(X,\CO_C(-1))\oplus H^1(X,\CO_C(-n-1)) \cong \BC^n$. 
\end{proof}

Now we will show that the extension class $\mathcal{E}$ is nonzero. Further, we will study the map $i \colon C \rightarrow \Ext^1(\CO_C(k+1),\CO_C(k)[1])$ induced by $\mathcal{E}=\CO_S$. We will do computations on the local model described in Lemma \ref{lem:local} and Lemma \ref{lem:local2}.

\begin{lem}
	\label{lem:degree}
	The degree of the map $i \colon C \hookrightarrow \BP^{n-1}$ is $n-1$.
\end{lem}
\begin{proof}
	The family $\CO_S$ induces a map from $C$ to $\Ext^1(\CO_C(k+1),\CO_C(k)[1])$, which we can see via the following cohomology argument. We can compute the cohomology of the pullback $\BL j^*j_* \CO_{C\times C}(-k-2,k)$, using the fact that $j$ is the inclusion of a divisor in $X \times X$. $H^0(\BL j^*j_* \CO_{C\times C}(-k-2,k))=\CO_{C\times C}(-k-2,k)$, and $H^{-1}(\BL j^*j_* \CO_{C\times C}(-k-2,k))=\CO_{C\times C}(n-k-2,k)$. This shows that $H^0(\BL j^*\CE)=\CO_{C\times C}(-k-1,k+1)$ and $H^{-1}(\CO_{C\times C}(n-k-2,k))$.
	
	$\CE$ is then a class in $\Ext^1(\CO_{C \times C}(-k-1,k+1),\CO_{C \times C}(n-k-2,k)[1])$. This space is isomorphic to $H^0(\CO_C(n-1))\otimes \Ext^1(\CO_C(k+1),\CO_C(k)[1])$. The map that $C$ induces to $\Ext^1(\CO_C(k+1),\CO_C(k)[1])$ comes from a section of $\CO_C(n-1)$. As long as this section is nonzero, this map has degree $n-1$. We will now show this section is nonzero. 
	
	Let $c$ be a point on $C$. Consider the inclusion $i_c\colon c\times X \rightarrow X \times X$. We will now show that $i_c^* \CE$ is a non-split extension of $\CO_C(-k-1,k+1)$ and $\CO_C(-k-2,k)[1]$. Lemma \ref{lem:local} shows we can do this computation on the local model. 
	
	By Lemma \ref{lem:local2}, we can see that the coordinate ring of $\widehat{X} \times \widehat{X}$ is the completion of the ring $R=\BC[x_1,y_1,w_1,x_2,y_2,w_2]$ with respect to $w_1$ and $w_2$, where $w_1$ and $w_2$ are the equations of the curve in each component, and have degree $(-n,0)$ and $(0,-n)$ respectively. The degree of $x_1$ and $y_1$ will be $(1,0)$, and the degree of $x_2$ and $y_2$ will be $(0,1)$. 
	
	Using the description of $\CE \cong \CO_S$, where $S=(C \times C)\cup_{\Delta_C} \Delta_{\widehat{X}}$, we can write down a free resolution of $\CE$, which we will then pull back via $i_C$.  $S$ is defined in $\widehat{X}$ by the ideal $(w_1(x_1y_2-x_2y_1), w_2(x_1y_2-x_2y_1),x_1^nw_1-x_2^nw_2,\dots , y_1^nw_1-y_2^nw_2)$. The resolution of this ideal is 
	$$  R^{\oplus n-1} \rightarrow R^{\oplus 2n+1} \rightarrow R^{\oplus n+3} \rightarrow R. $$
	
	Pulled back to $c \times X$, and considering degrees, this gives a resolution of $\BL i_c^* \CE$ as follows. 
	\begin{center}
		\begin{tikzpicture}[scale=0.5]
		\matrix (m) [matrix of math nodes,row sep=3em,column sep=2em,minimum width=2em]
		{
			\CO(k)^{\oplus n-1} & \CO(k+n)\oplus (\CO(k)\oplus \CO(k+1))^{\oplus n} & \CO(k) \oplus \CO(k+n) \oplus \CO(k+1)^{\oplus n+1} & \CO(k+1) \\};
		\path[-stealth]
		(m-1-1) edge node [above] {$M_3$}  (m-1-2)
		(m-1-2) edge node [above] {$M_2$} (m-1-3)
		(m-1-3) edge node [above] {$M_1$} (m-1-4);
		\end{tikzpicture}
	\end{center}
	The maps in this sequence are in terms of $x_2,y_2,w_2$. $x_1$ and $y_1$ are fixed. The first map $M_1$ is $$M_1=\left( \begin{array}{ccccc}
	0 & w_2(x_1y_2-x_2y_1) & x_2^nw_2 & \cdots & y_2^nw_2 \\
	\end{array} \right).$$
	The next map $M_2$ is given by
	$$M_2=\left( \begin{array}{cccccc}
	-w_2 & x_1^{n-1} & 0 & \cdots & y_1^{n-1} & 0 \\
	0 & 0 & x_2^{n-1} &  \cdots & 0& y_2^{n-1}  \\
	0 & -y_2 & -y_1 & \cdots & 0 & 0 \\
	0 & x_2 & x_1 & \cdots & 0 & 0 \\
	\vdots & \vdots & \vdots & & \vdots & \vdots \\
	0 & 0 & 0 & \cdots & -y_2 & -y_1 \\
	0 & 0 & 0 & \cdots & x_2 & x_1 \\
	\end{array} \right).$$
	The last map $M_3$ is given by 
	$$M_3=\left( \begin{array}{cccc}
	0 & 0 & \cdots & 0 \\
	-y_1 & 0 & \cdots & 0 \\
	y_2 & 0 & \cdots & 0 \\
	x_1 & -y_1 & \cdots & 0 \\
	-x_2 & y_2 & \cdots & 0 \\
	\vdots & \vdots & & \vdots \\
	0 & 0 & \cdots  & x_1 \\
	0 & 0 &  \cdots & -x_2 \\
	\end{array} \right).$$
	
	Let $F\cx$ be the resolution described above. Recall the notation $\tau_{\leq a}$ given in Section \ref{sec:back}. There is an exact triangle $$\tau_{\leq -1} F\cx \rightarrow F\cx \rightarrow H^0(F\cx).$$ Since $\Hom(\CO_C(k)[1],H^0(F\cx))=0$, and $\Hom(\CO_C(k)[1],H^0(F\cx)[-1])=0$, we know from the long exact Hom sequence applied to the triangle that $\Hom(\CO_C(k)[1],F\cx) \cong \Hom(\CO_C(k)[1],\tau_{\leq -1} F\cx)$.
	
	Similarly, there is an exact triangle $$\tau_{\leq -2}F\cx \rightarrow \tau{\leq -1}F\cx \rightarrow H^{-1}(F\cx)[1].$$ By degree arguments, $\Hom(\CO_C(k)[1],\tau_{\leq -2}F\cx)=0$. It remains to compute $\Hom(\CO_C(k)[1],H^{-1}(F\cx)[1])$.  
	
	We know that $H^{-1}(F\cx) \cong {\rm Ker}(M_1)/{\rm Im}(M_2)$. Looking at the maps $M_2$ and $M_1$ explicitly, we see this quotient is supported on $C$, in degree higher than $k$. Hence, there are no morphisms from $\CO_C(k)[1]$ to this resolution, and so the sequence $0 \rightarrow \CO_C(k)[1] \rightarrow \BL i_C^* \CE \rightarrow \CO_C(k+1) \rightarrow 0$ is non-split 
\end{proof}

\begin{prop}
	\label{prop:bijection} Let $\eta \colon \BP^{n-1}\rightarrow \BP \Ext^1(\CO_C(k+1),\CO_C(k)[1])$ be the isomorphism described in Lemma \ref{lem:exts}. 
	There is a bijection $\gamma \colon (X-C) \cup \BP^{n-1} \rightarrow M_{\tau}([\CO_x])$ defined as follows:
	$$\gamma(y)=
	\begin{cases}
	\CO_y, & {\rm if} \ y \in X\setminus C \\
	\eta(y), & {\rm if} \ y \in \BP^{n-1}. \\
	\end{cases}$$
	
\end{prop}
\begin{proof}
	This follows from Lemma \ref{prop:otherx} and Lemma \ref{lem:exts}. 
\end{proof}

\begin{prop}
	Let $Y=X\sqcup_C \BP^{n-1}$. Then there is a family $\mathcal{U}_{\tau}$ of $\tau$-stable objects over $Y$ such that the induced map $Y \colon M_{\tau}([\CO_x])$ induces the injection in Proposition \ref{prop:bijection} on points.
\end{prop}

\begin{proof}
	We have constructed a family on $X$, the object $\CO_S$ in $\CDb(X \times X)$. We also have a family on $\BP^{n-1}$ given by the universal extension of $\CO_C(k+1)$ and $\CO_C(k)[1]$ \cite[p.118]{LP}. Consider the projections $p_1 \colon \BP^{n-1} \times X \rightarrow \BP^{n-1}$ and $p_2 \colon \BP^{n-1} \times X \rightarrow X$. By \cite[Proposition 4.2.2]{LP}, the object $\mathcal{E}{\it xt}^1(p_2^*\CO_C(k+1), p_1^*\CO_C(k)[1])$ in $\CDb(\BP^{n-1} \times X)$ is isomorphic to $H^0(\CO_{\BP^{n-1}} \otimes \Ext^1(\CO_C(k+1),\CO_C(k)[1]))$. If we consider the element $E\cx \in  \Ext^1(p_2^*\CO_C(k+1), p_1^*\CO_C(k)[1])$ in $\CDb(\BP^{n-1} \times X)$ corresponding to the identity map $\BP^{n-1} \rightarrow \BP^{n-1}$, then $E\cx \in \CDb(\BP^{n-1} \times X)$ is a universal family on $\BP^{n-1}$ parametrizing extensions $\Ext^1(\CO_C(k+1),\CO_C(k)[1])$.

	We now claim that these two objects can be glued along $C$ to construct a family $\mathcal{U}_{\tau}$ over $Y=(X\setminus C) \sqcup \mathbb{P}^{n-1}$ inducing the injection in Proposition \ref{prop:bijection}. 
	
	Consider the following diagram of inclusions:
	\begin{center}
		\begin{tikzpicture}
		\matrix (m) [matrix of math nodes,row sep=3em,column sep=4em,minimum width=2em]
		{
			C \times X &  X \times X \\
			\BP^{n-1} \times X & Y \times X. \\};
		\path[-stealth]
		(m-1-1) edge node [left] {$i_2$} (m-2-1)
		(m-1-1) edge node [above] {$i_1$} (m-1-2)
		(m-2-1) edge node [below] {$j_2$} (m-2-2)
		(m-1-2) edge node [right] {$j_1$} (m-2-2)
		(m-1-1) edge node [above] {$i$} (m-2-2);
		\end{tikzpicture}
	\end{center}
	By construction, $i_1^*\CO_S \cong i_2^*E\cx$. Let $L\cx=i_1^*\CO_S$. Further, if we consider the isomorphisms, $i_1^*\CO_S \rightarrow L\cx$ and $i_2^*E\cx \rightarrow L\cx$, then via adjunction and push-forward we get morphisms $(j_1)_*\CO_S \rightarrow i_*L\cx$ and $(j_2)_*E\cx \rightarrow i_*L\cx$. Define $P\cx$ to be the object fitting into the exact triangle
	$$ P\cx \rightarrow (j_1)_*\CO_S \oplus (i_2)_*E\cx \rightarrow i_*L\cx.$$ We will now show that $P\cx$ is the desired family $\mathcal{U}_{\tau}$. 
	
	First, suppose $x$ is a point in $X \setminus C$. Then restricting the triangle above to $\{x\} \times X$, $E\cx$ and $L\cx$ become $0$, so $P\cx\vert_{\{x\} \times X} \cong \CO_S\vert_{\{x\} \times X}$. Similarly, if we choose a point $y \in \mathbb{P}^{n-1}\setminus C$, we find $P\cx\vert_{\{y\} \times X} \cong E\cx \vert_{\{y\}\times \BP^{n-1}}$. What remains is to show that $P\cx\vert_{C\times X} \cong L\cx$. In fact, we will show that this is true in a formal neighbourhood of a point $x \in C$. That is, we will look at the exact triangle
	$$P\cx \otimes^{\mathbb{L}} i_*\CO_C \rightarrow ((j_1)_*\CO_S \oplus (j_2)_*E\cx)\otimes^{\mathbb{L}} i_*\CO_C \rightarrow i_*L\cx\otimes^{\mathbb{L}} i_*\CO_C$$ and show that in a formal neighbourhood of a point $x \in C$, $P\cx \otimes^{\mathbb{L}}i_*\CO_C \cong i_*L\cx$. This will show via the projection formula that near $x$, $i_*i^*P\cx \cong i_*L\cx$ so that $i^*P\cx \cong L\cx$.
	
	We will now describe a formal neighbourhood of a point $x \in C$. Along $C$, we can look at an affine patch $\BA^{n-1}$ of $\BP^{n-1}$ and $\BA^2$ of $X$, glued along the affine patch $\BA^1$ of $C$. The coordinate ring of this space is $R=k[x,y,z_1, \dots, z_{n-2}]/(yz_1, \dots, yz_{n-1})$, where $C={\rm Spec}(k[x])$, $\BA^2={\rm Spec}(k[x,y])$ and $\BA^{n-1}={\rm Spec}(k[x,z_1, \dots, z_{n-2}])$. The formal neighbourhood of $x$ in $Y$ is the completion $k[[x,y,z_1,\dots, z_{n-2}]]/(yz_1,\dots, yz_{n-2})$ of this ring $R$. Since the inclusion of this neighbourhood in $Y$ is flat, we may restrict any complexes to this neighbourhood Here on we will use $i_1$, $i_2$, $j_1$, $j_2$, and $i$ to describe these maps after base change. 
	
	The resolution of $C$ in the ring $R$ is the resolution of the ideal $(y,z_1, \dots, z_{n-1})$. This resolution is given by the complex $R\cx$ below.
	\begin{center}
		\begin{tikzpicture}[scale=0.5]
		\matrix (m) [matrix of math nodes,row sep=3em,column sep=2em,minimum width=2em]
		{
			\cdots & R^{a_2} & R^{a_1} & R^{a_0}. \\};
		\path[-stealth]
		(m-1-1) edge node [above] {$d_2$}  (m-1-2)
		(m-1-2) edge node [above] {$d_1$} (m-1-3)
		(m-1-3) edge node [above] {$d_0$} (m-1-4);
		\end{tikzpicture}
	\end{center} 
	We can see that $a_0=1$ and $a_1=n-1$, as the first differential, $d_0$ is given by multiplication by the equations $y, z_1, \dots, z_{n-1}$ describing $C$. The next differential, $d_1$, describes the relations between these. The relations are given by the $2(n-2)$ products of $z_i$ with $y$ (which is $0$ in this ring) and the first step in the Koszul complex for $z_1, \dots, z_{n-2}$, call it $K\cx$. This gives $a_3=2(n-2)+\binom{n-2}{2}=\frac{(n-2)(n+1)}{2}$ factors of $R$ at the third step in the resolution. For example, when $n=4$, the degree $-2$ to $0$ terms of $R\cx$ are
	\begin{center}
		\begin{tikzpicture}[scale=1.5]
		\matrix (m) [matrix of math nodes,row sep=3em,column sep=10em,minimum width=2em]
		{
			\cdots & R^{5} & R^{3} & R. \\};
		\path[-stealth]
		(m-1-1) edge node [above] {$d_2$}  (m-1-2)
		(m-1-2) edge node [above] {$\left( \begin{array}{ccccc}
			0 & z_2 & 0 & z_1 & 0 \\
			-z_2 & 0 & 0 & 0 & y \\
			z_1 &0 & y & 0 & 0\\
			\end{array} \right)$} (m-1-3)
		(m-1-3) edge node [above] {$\left( \begin{array}{ccc}
			y & z_1 & z_2 \\
			\end{array} \right)$} (m-1-4);
		\end{tikzpicture}
	\end{center}

	Let $m_i$ be the rank of $R$ in the $i$th term of the resolution. The $i$th differential will consist of linear terms $y, z_1, \dots, z_{n-2}$ which multiply with the $i-1$th differential to give relations of the form $yz_i$ or relations in the Koszul complex of $z_1, \dots, z_{n-2}$. For every summand $R$ of $R^{a_i}$, the total number of the maps coming into $R$ from $d_i$ and the maps coming out of $R$ from $d_{i-1}$ will be $n-1$, with each linear map $y,z_1, \dots, z_{n-2}$ appearing exactly once. 
	
	We now compute the tensor product $(j_1)_*\CO_S \otimes^{\mathbb{L}} i_*\CO_C$, obtained by tensoring $R\cx$ with $(j_1)_*\CO_S$. The maps $z_i$ are $0$ on $\CO_S$, since $\CO_S$ is supported on $X \times X$ and $z_i$ are coordinated on $\BP^{n-1}$. Furthermore, any term of the form $\bigl((j_1)_*\CO_S \xrightarrow{y} (j_1)_*\CO_S\bigr)$ is isomorphic to $i_*L\cx$. This follows from the fact that $(i_1)_*L\cx \cong \CO_S \otimes^{\mathbb{L}} (i_1)_*\CO_C \cong \bigl((j_1)_*\CO_S \xrightarrow{y} (j_1)_*\CO_S\bigr)$ via the map $y$. Therefore $i_*L\cx \cong (j_1)_*\CO_S \rightarrow (j_1)_*\CO_S$ via the map $y$.
	
	Every copy of $(j_1)_*\CO_S$ in the complex $(j_1)_*\CO_S \otimes^{\mathbb{L}} i_*\CO_C$ given by		\begin{center}
		\begin{tikzpicture}[scale=0.5]
		\matrix (m) [matrix of math nodes,row sep=3em,column sep=2em,minimum width=2em]
		{
			\cdots & (j_1)_*\CO_S^{a_2} & (j_1)_*\CO_S^{a_1} & (j_1)_*\CO_S^{a_0}. \\};
		\path[-stealth]
		(m-1-1) edge node [above] {$d_2$}  (m-1-2)
		(m-1-2) edge node [above] {$d_1$} (m-1-3)
		(m-1-3) edge node [above] {$d_0$} (m-1-4);
		\end{tikzpicture}
	\end{center} 
	will occur either at the end of an incoming map $y$ or at the beginning of an outgoing map $y$. Therefore, this complex is isomorphic to $$(j_1)_*\CO_S \otimes^{\mathbb{L}} i_*\CO_C \cong i_*L\cx \oplus i_*L\cx[1]^{\oplus b_1} \oplus i_*L\cx[2]^{\oplus b_2} \oplus \cdots$$ where $b_i$ is the number of incoming $y$ maps in the $-i$th degree term of $R\cx$. 
	
	Now, we compute the the tensor product of $(j_2)_*E\cx \otimes^{\mathbb{L}} i_*\CO_C$.  The maps $y$ are $0$ on $(j_2)_*E\cx$, since $E\cx$ is supported on $\BP^{n-1}\times X$ and $y$ is a coordinate on $X$. Further, the Koszul complex $K\cx$ of $z_1, \dots, z_{n-2}$ tensored with $(j_2)_*E\cx$ is isomorphic to $i_*L\cx$. This is because $(i_2)_*L\cx \cong E\cx \otimes i_*\CO_C \cong E\cx \otimes K\cx$. Hence $i_*L\cx \cong (j_2)_*E\cx \otimes K\cx$. 
	
	Consider the complex $(j_2)_*E\cx \otimes^{\mathbb{L}} i_*\CO_C$ given by		\begin{center}
		\begin{tikzpicture}[scale=0.5]
		\matrix (m) [matrix of math nodes,row sep=3em,column sep=2em,minimum width=2em]
		{
			\cdots & (j_2)_*(E\cx)^{\oplus a_2} & (j_2)_*(E\cx)^{\oplus a_1} & (j_2)_*E\cx. \\};
		\path[-stealth]
		(m-1-1) edge node [above] {$d_2$}  (m-1-2)
		(m-1-2) edge node [above] {$d_1$} (m-1-3)
		(m-1-3) edge node [above] {$d_0$} (m-1-4);
		\end{tikzpicture}
	\end{center} For each copy of $(j_2)_*E\cx$ which occurs in degree $-i$ in this complex, and occurs at the end of a complex of the form $(j_2)_*E\cx \otimes K\cx$, we get a direct summand of $i_*L\cx[i]$ in $(j_2)_*E\cx \otimes^{\mathbb{L}}i_*\CO_C$. Since the maps $y$ are now $0$, all nonzero maps in this complex occur as part of some shift of $(j_2)_*E\cx \otimes K\cx$, hence $(j_2)_*E\cx \otimes^{\mathbb{L}}i_*\CO_C$ is a direct sum of shifts of $i_*L\cx$. We must now count these terms to determine the complex $(j_2)_*E\cx \otimes^{\mathbb{L}}i_*\CO_C$.
	
	If we now consider the complex $R\cx$, we note that if a summand $R$ of $R^{a_i}$ has an outgoing map $y$, then it must have $n-2$ incoming maps $z_1, \dots, z_{n-2}$, since $yz_i=0$ is a relation in $R$. We have seen that the differentials in $R\cx$ all come from relations $yz_i$ or from the differentials in $K\cx$, A summand $R$ with $n-2$ incoming maps, one for each $z_i$, must then occur at the end of a Koszul complex $K\cx$. Therefore, if we let $c_i$ be the degree of $i_*L\cx[i]$ in $(j_2)_*E\cx \otimes^{\mathbb{L}}i_*\CO_C$, we see $c_i=a_i-b_i$ for $i>0$. 
	
	Now consider $i_*L\cx \otimes i_*\CO_C$. The maps $y, z_1, \dots, z_{n-2}$ are all $0$ on $C$. Hence this complex will be a direct sum of terms $i_*L\cx$, of the form $i_*L\cx \oplus i_*L\cx[1]^{\oplus n-1} \oplus i_*L\cx[2]^{\oplus a_1} \oplus \cdots$. The exact triangle $$P\cx \otimes^{\mathbb{L}} i_*\CO_C \rightarrow ((j_1)_*\CO_S \oplus (j_2)_*E\cx)\otimes^{\mathbb{L}} i_*\CO_C \rightarrow i_*L\cx\otimes^{\mathbb{L}}i_*\CO_C$$ is now locally given by
	$$P\cx \otimes^{\mathbb{L}}i_*\CO_C \rightarrow (i_*L\cx \oplus i_*L\cx[1]^{\oplus b_1} \oplus \cdots)\oplus(i_*L\cx \oplus i_*L\cx[1]^{\oplus a_1-b_1} \oplus \cdots)\rightarrow i_*L\cx \oplus i_*L\cx[1]^{\oplus a_1}\oplus \cdots. $$ We can then see that $P\cx \otimes^{\mathbb{L}} i_*\CO_C \cong i_*L\cx$, which completes our proof that $P\cx$ is a family over $Y$ which is obtained by gluing $E\cx$ and $\CO_S$ along $C$. Since $E\cx$ and $\CO_S$ induce the isomorphism in Proposition \ref{prop:bijection} over $\BP^{n-1}$ and $X$ respectively, and agree on $C$, the glued object $P\cx$ we have constructed will induce the map $Y \rightarrow M_{\tau}([\CO_x])$ in Proposition \ref{prop:bijection}.
\end{proof}

We now will describe the tangent space $\Ext^1(E,E)$ for $E \in M_{\tau}([\CO_x])$. In particular, we will show that $\gamma$ induces an isomorphism of tangent spaces between $M_{\tau}([\CO_x])$ and $X \sqcup_C \BP^{n-1}$, where $C$ is embedded as a rational normal curve in $\BP^{n-1}$. In the course of this argument we will specifically describe the image of $C$ in $\BP \Ext^1(\CO_C(k+1),\CO_C(k)[1])$ as the locus of extensions where the tangent space jumps in dimension. 

\begin{lem}
	\label{lem:tan}
	The map $\gamma \colon X \sqcup_C \BP^{n-1} \rightarrow M^{\tau}([\CO_x])$ induces a isomorphism of tangent spaces. 
\end{lem}

\begin{proof}
	If $E \in  M^{\tau}([\CO_x])$ is the class of a stable object $\CO_x$ for some $x \in X\setminus C$ then $\Ext^1(E,E) \cong T_xX$. We will now consider $E$ a stable object of class $[\CO_x]$ for some $x \in C$. We will show that $\Ext^1(E,E) \cong \Ext^1(\CO_C(k+1),\CO_C(k)[1]) \cong \BC^{n-1}$ except for $E$ lying on a copy of $C$ in $\BP^{n-1}$, where $\Ext^1(E,E) \cong \BC^n$. 
	
	There is a morphism $\Ext^1(\CO_C(k+1),\CO_C(k)[1]) \rightarrow \Ext^1(E,E)$ given by composing $\Ext^1(\CO_C(k+1),\CO_C(k)[1]) \rightarrow \Ext^1(E,\CO_C(k)[1]) \rightarrow \Ext^1(E,E)$. We would like to show this morphism is surjective. Applying Hom to the triangle $\CO_C(k)[1] \rightarrow E \rightarrow \CO_C(k+1)$, we see this is equivalent to showing that $\Ext^1(E, \CO_C(k+1)) \rightarrow \Ext^2(E,\CO_C(k)[1])$ is injective. 
	
	Consider the commutative diagram of exact sequences: 
	
	\begin{center}
		\begin{tikzpicture}
		\matrix (m) [matrix of math nodes,row sep=3em,column sep=4em,minimum width=2em]
		{
			\Ext^2(\CO_C(k+1),\CO_C(k+1)) & 0 \\
			\Hom(\CO_C(k),\CO_C(k+1)) & \Ext^2(\CO_C(k),\CO_C(k)) \\
			\Ext^1(E,\CO_C(k+1)) & \Ext^2(E,\CO_C(k)[1]) \\
			0& 0 \\};
		\path[-stealth]
		(m-3-1) edge node [above] {$\alpha$}  (m-3-2) edge (m-2-1)
		(m-2-1) edge node [left] {$\lambda$} (m-1-1) edge node [above] {$\beta$} (m-2-2)
		(m-3-2) edge (m-2-2)
		(m-2-2) edge (m-1-2)
		(m-1-1) edge  (m-1-2)
		(m-4-1) edge (m-3-1) edge (m-4-2)
		(m-4-2) edge (m-3-2);
		\end{tikzpicture}
	\end{center}
	In order to show that the map $f$ is injective, we show that ${\rm ker}(\beta)$ and ${\rm ker}(\lambda)$ intersect non-trivially in $\Hom(\CO_C(k),\CO_C(k+1))$. 
	
	Let $\Delta \in \Hom(\CO_C(k+1),\CO_C(k)[2])$ be the class of $E$. For $f \in \Hom(\CO_C(k),\CO_C(k+1))$, $\beta(f)= \Delta \circ f$. By Serre duality, we have isomorphisms $\phi$ so that the following square is commutative.
	\begin{center}
		\begin{tikzpicture}
		\matrix (m) [matrix of math nodes,row sep=3em,column sep=4em,minimum width=2em]
		{
			\Hom(\CO_C(k+1),\CO_C(k)[2]) & \Hom(\CO_C(k),\CO_C(k)[2]) \\
			\Hom(\CO_C(k)[2],\CO_C(k+n-1)[2])^* & \Hom(\CO_C(k)[2],\CO_C(k+n-2)[2])^* \\};
		\path[-stealth]
		(m-1-1) edge node [above] {$f$} (m-1-2) 
		(m-1-2) edge node [right] {$\phi$} (m-2-2)
		(m-1-1) edge node [left] {$\phi$} (m-2-1) 
		(m-2-1) edge node [above] {$F$} (m-2-2);
		\end{tikzpicture}
	\end{center}
	The map $f$ composes a class $\Delta$ with $f$. For any $g \in \Hom(\CO_C(k)[2],\CO_C(k+n-2)[2])$ and functional $\xi \in \Hom(\CO_C(k)[2],\CO_C(k+n-1)[2])^*$, $F(\xi)(g)=\xi(f[2] \circ g)$, where $f[2]$ is now viewed as lying in $\Hom(\CO_C(k+n-2)[2],\CO_C(k+n-1)[2])$. The commutativity of this square shows that $\phi(\Delta \circ f)(g)=\phi(\Delta)(f[2]\circ g)$.
	
	Similarly, $\lambda(f)=f \circ \Delta$. Using Serre duality, we see that for $h \in \Hom(\CO_C(k+1),\CO_C(k+n-1))$, $\phi(\lambda(f))(h)=\phi(f \circ \Delta)(h)=\phi(\Delta)(h \circ f[2])$. We now see that ${\rm ker}(\lambda)={\rm ker}(\beta)$, and is given by the condition that $f$ must be such that $\phi(\Delta)$ vanishes on the image of the map $\Hom(\CO_C(k)[2],\CO_C(k+n-2)[2]) \rightarrow \Hom(\CO_C(k)[2],\CO_C(k+n-1)[2])$ given by multiplication by $f[2]$. For general $\Delta$, no such $f$ will exist, and both $\beta$ and $\lambda$ will be  injective. In this case, there is a surjection $ \Ext^1(\CO_C(k+1), \CO_C(k)[1]) \twoheadrightarrow \Ext^1(E,E)$ with $1$-dimensional kernel, and $\Ext^1(E,E) \cong \BP^{n-1}$.
	
	For each point $c$ on a rational curve $C$, there is a $\Delta_c \in \Hom(\CO_C(k+1),\CO_C(k)[2])$ which is dual to $\delta_c \in \Hom(\CO_C(k)[2],\CO_C(k+n-1)[2])$, the shift by 2 of the map $\CO_C(k) \rightarrow \CO_C(k+n-1)$ given with cokernel supported at $c$. For this $\Delta_c$, the kernel of $\beta$ and the kernel of $\lambda$ is one-dimensional, and $\Ext^1(E,E) \cong \BP^n$. 
\end{proof}

We have shown in Proposition \ref{prop:bijection} that $\gamma$ is a bijection on points, and in Lemma \ref{lem:tan} that $\gamma$ induces an isomorphism of tangent spaces. Were $X \sqcup_C \BP^{n-1}$ smooth, then following \cite[Corollary 14.10]{Harris} this would be enough to show that $\gamma$ is an isomorphism. Of course, $X \sqcup_C \BP^{n-1}$ is not smooth when $n>2$. It is in fact reducible, singular along the curve $C$ where the two varieties $X$ and $\BP^{n-1}$ meet. Hence \cite[Corollary 14.10]{Harris} is enough to show only that $\gamma$ is an isomorphism away from $C$. 

However, the proof of \cite[Corollary 14.10]{Harris} does not require smoothness. In fact, in our case, the only concern we might have is that without smoothness, the map $\gamma^*$ might not be injective, which is required in Harris' proof. The following lemma will show that in fact we need only to show that $M_{\tau}([\CO_x])$ is reduced in order to show that $\gamma^*$ is injective and apply $\cite[Corollary 14.10]{Harris}$. The work of showing that $M_{\tau}([\CO_x])$ is reduced is the content of Section \ref{sec:reducedness}.

\begin{lem}
\label{lem:injred}
Let $\pi \colon X \rightarrow Y$ be a surjective morphism of affine varieties, and let $Y$ be reduced. Then the induced ring homomorphism is injective.
\end{lem}

\begin{proof}
Say $X= {\rm Spec}(B)$ and $Y={\rm Spec}(A)$. Suppose $a$ is in the kernel of $\pi^*$. Then $\pi^*(a)=0$, which lies in every prime ideal of $B$. Since the map $\pi$ is surjective, this implies that $a$ lies in every prime ideal of $A$. Therefore $a$ is in the nilradical of $A$. Since $A$ is reduced, $a=0$. 
\end{proof}

\section{On the reducedness of the moduli space $M_{\tau}([\mathcal{O}_{x}])$}
\label{sec:reducedness}
Since we are going to study local properties of the moduli space $M_{\tau}([\mathcal{O}_{x}])$ of complexes, we start with some definitions and properties from the deformation theory. Let $Y$ be a smooth projective variety and $E\in\mathrm{D}^{\mathrm{b}}(Y)$ be a complex in its bounded derived category. Let $\mathrm{Art}$ be the category of Artin local ring over $\mathbb{C}$ and $A\in\mathrm{Art}$.
\begin{defn}
A deformation of $E$ over $A$ is a complex $E_{A}\in\mathrm{D}^{\mathrm{b}}(Y_{A})$, where $Y_{A}=Y\times\mathrm{Spec}A$, such that the derived pullback of $E_{A}$ to the closed fiber $Y_{0}=Y\times\{0\}$ is $E$. In particular, if we take $A=\mathbb{C}[\varepsilon]/\varepsilon^{2}$, we call it a first order deformation of $E$. A first order deformation $E_{1}$ can be lifted to the second order if there exists a deformation $E_{2}$ over $\mathbb{C}[\varepsilon]/\varepsilon^{3}$ extending $E_{1}$ via the natural closed embedding.
\end{defn}
\begin{prop}
The first order deformations of $E$ are parametrized by 
\begin{equation*}
\mathrm{Ext}^{1}(E,E):=\mathrm{Hom}(E,E[1])
\end{equation*}. The first order deformations which can be lifted to the second order are parametrized by $\kappa_{2}^{-1}(0)\subseteq\mathrm{Ext}^{1}(E,E)$, where \begin{equation*}
\kappa_{2}:\mathrm{Ext}^{1}(E,E)\rightarrow\mathrm{Hom}(E,E[2])
\end{equation*} sends $\xi\in\mathrm{Ext}^{1}(E,E)$ to $\xi[1]\circ\xi$.
\end{prop}
\begin{proof}
This is well-known for sheaves, for example, one can see \cite{DMCh06}. It is carried over to the case of complexes by \cite{Lie06}.
\end{proof}

Now assume that $E$ is a stable complex with respect to some stability condition $\sigma$ not lying on any wall inside the stability manifold (assume this is non-empty), we associate to $E$ its deformation functor \begin{equation*}
\mathrm{Def}_{E}: \mathrm{Art}\rightarrow \mathrm{Set}
\end{equation*}
by sending an Artin local ring $A$ to the set of all deformations of $E$ over $A$. In \cite{Lie06}, it is proved that this functor satisfies the first three conditions of Schlessinger's criterion, guaranteeing the existence of a hull for $\mathrm{Def}_{E}$. In our case the complex $E$ is stable, its automorphisms are just scalar multiplication by a non-zero constant, therefore they always extend via small thickenings. This proves the last condition of Schlessinger's condition, hence $\mathrm{Def}_{E}$ is prorepresentable. If a moduli space for $E$ exists, the complete local ring $R$ that prorepresents $\mathrm{Def}_{E}$ will become the completion of the local ring of the moduli space at $E$. This complete local ring $R$ can be computed explicitly via the so-called Kuranishi map, which we now describe. The Kuranishi map is a formal map
\begin{equation*}
\kappa=\kappa_{2}+\kappa_{3}+\dots :\mathrm{Ext}^{1}(E,E)\rightarrow\mathrm{Ext}^{2}(E,E)
\end{equation*}
defined inductively on order by using obstruction theory. An explicit construction can be found in the appendix A of \cite{LS06} in the case of sheaves, it applies to the case of complexes in the same way. The formal scheme $\kappa^{-1}(0)$ parametrizes all the versal deformations of $E$ and satisfies certain universal properties. As a result, this is the desired hull of our deformation functor $\mathrm{Def}_{E}$ and $\kappa^{-1}(0)=\mathrm{Spec}R$ as a formal scheme, for more details one can see chapter 3 of \cite{AS18} in the case of sheaves.

Going back to our situation, let $E$ be a complex in the image of $C$ under $\gamma$ and $R$ be the completion of the local ring of $M_{\tau}([\mathcal{O}_{x}])$ at $E$, we only need to show that $R$ is reduced. The strategy is the following: First, we will compute $\kappa_{2}$ with the help of destablizing sequences and Proposition 8.2, we will see that $\kappa_{2}^{-1}(0)$ is cut out by quadratic equations which are products of different linear forms, hence is a union of a two dimensional subspace and an $n-1$ dimensional subpaces; Then we argue that by the construction of the morphism $\gamma$, we should have two tangent spaces $T_{\mathbb{P}^{n-1},E}$ of dimension $n-1$ and $T_{X,E}$ of dimension two lying inside $\kappa^{-1}(0)$, which in particular lying inside $\kappa_{2}^{-1}(0)$. Hence there is no room for other possibilities, we must have that $\kappa^{-1}(0)=\kappa_{2}^{-1}(0)$ is cut out by quadratic equations which are products of different linear forms. This proves that $\kappa^{-1}(0)$ is reduced and therefore $R$ is reduced.

First we denote the arrows in the destablizing sequence (in the proof of Lemma 7.6) by \begin{equation*}
\oks\overset{a}{\longrightarrow}E\overset{b}{\longrightarrow}\okp\overset{\eta}{\longrightarrow}\mathcal{O}_{C}(k)[2].
\end{equation*}
By writing down long exact sequence for Hom functor, we will have the following commutative diagrams (for simplicity we denote $A$ to be $\oks$ and $B$ to be $\okp$).
\begin{lem}\label{lem:erc}
The following diagram has exact rows and columns except at $\mathrm{Ext}^{1}(B,A)$ where we have a common one-dimensional kernel $\mathbb{C}\eta$:
\medskip
\begin{center}
$\begin{CD}
\mathrm{Ext}^{1}(B,A)=\mathbb{C}^{n} @>>> \mathrm{Ext}^{1}(E,A)=\mathbb{C}^{n-1} @>>> 0\\
@VVV @VVV @VVV\\
\mathrm{Ext}^{1}(B,E)=\mathbb{C}^{n-1} @>>> \mathrm{Ext}^{1}(E,E)=\mathbb{C}^{n} @>>> \mathrm{Ext}^{1}(A,E)=\mathbb{C}\\
@VVV @VVV @VVV\\
0 @>>> \mathrm{Ext}^{1}(E,B)=\mathbb{C} @>>> \mathrm{Ext}^{1}(A,B)=\mathbb{C}^{2}\\
@VVV @VVV @VVV\\
0 @>>> \mathrm{Ext}^{2}(E,A)=\mathbb{C}^{n-1} @>>> \mathrm{Ext}^{2}(A,A)=\mathbb{C}^{n-1}\\
@VVV @VVV @VVV\\
\mathrm{Ext}^{2}(A,E)=\mathbb{C}^{n-1} @>>> \mathrm{Ext}^{2}(E,E)=\mathbb{C}^{2n-3} @>>> \mathrm{Ext}^{2}(A,E)=\mathbb{C}^{n-2}\\
@VVV @VVV @VVV\\
\mathrm{Ext}^{2}(B,B)=\mathbb{C}^{n-1} @>>> \mathrm{Ext}^{2}(E,B)=\mathbb{C}^{n-2} @>>> 0
\end{CD}$
\end{center}
\medskip
\end{lem}
We can see the homomorphism $\theta:\mathrm{Ext}^{1}(E,E)\longrightarrow\mathrm{Ext}^{1}(A,B)$ sending an extension $\xi$ to $b[1]\circ\xi\circ a$ has an one-dimensional image in $\mathrm{Ext}^{1}(A,B)$. We can decompose $\mathrm{Ext}^{1}(E,E)=\mathrm{ker}(\theta)\oplus\mathbb{C}u$, where $u\in\mathrm{Ext}^{1}(E,E)$ satisfies $b[1]\circ u\circ a\neq0$. Notice that the homomorphism $\phi:\mathrm{Ext}^{1}(B,A)\longrightarrow\mathrm{Ext}^{1}(E,E)$ sending an arrow $\okp\overset{c}{\longrightarrow}\okss$ to $a[1]\circ c\circ b$ factors through $\mathrm{ker}(\theta)$. For dimension reason, we must have $\mathrm{im}(\phi)=\mathrm{ker}(\theta)$ and they both equal the tangent space of $\mathbb{P}^{n-1}$ at $E$. Since $E$ lies on $C$, we have a further decomposition $\mathrm{ker}(\theta)=N_{C/\mathbb{P}^{n-1},E}\oplus T_{C,E}$. Assume $N_{C/\mathbb{P}^{n-1},E}$ is generated by $\{v_{i}|i=1,2,\cdots,n-2\}$ and $T_{C,E}$ is generated by $w$. To summarise, given any $\xi\in\mathrm{Ext}^{1}(E,E)$, we can write it as
\begin{equation*}
\xi=au+bw+\Sigma_{i=1}^{n-2}c_{i}v_{i}
\end{equation*}
where $a$, $b$ and $c_{i}$ are coefficients. The next proposition computes $\kappa_{2}$ explicitly with respect to the bases chosen above.
\begin{prop}\label{prop:soo}
The second order obstruction map is computed by
\begin{equation*}
\kappa_{2}(\xi)=\sum_{i=1}^{n-2}ac_{i}(u[1]\circ v_{i}+v_{i}[1]\circ u),
\end{equation*}
and $\{u[1]\circ v_{i}+v_{i}[1]\circ u|i=1,2,\cdots,n-2\}$ are linearly independent in $\mathrm{Ext}^{2}(E,E)$
\end{prop}

We need one more lemma to prove the proposition. Since $E$ lies on $C$ which is contained in $X$, it will correspond to a point $x$ in $X$. We denote the arrows by
\begin{equation*}
B\overset{c}{\longrightarrow} \mathcal{O}_{x}\overset{d}{\longrightarrow} A \overset{e}{\longrightarrow} B[1].
\end{equation*}
By writing down long exact sequence for Hom functor, we will have the following commutative diagrams.
\begin{lem}
The following diagram is coming from the long exact sequences of $\mathrm{Hom}$ functor in two directions of the above extension. It is commutative, exact and all boundary homomorphisms are zero except at $\mathrm{Ext}^{1}(A,B)$, where we have a common one-dimensional kernel $\mathbb{C}e$.
\end{lem}
\begin{center}
$\begin{CD}
\mathrm{Ext}^{1}(A,B)=\mathbb{C}^{2} @>>> \mathrm{Ext}^{1}(A,\mathcal{O}_{x})=\mathbb{C} @>>> 0 @>>>0\\
@VVV @VVV @VVV @VVV\\
\mathrm{Ext}^{1}(\mathcal{O}_{x},B)=\mathbb{C} @>>> \mathrm{Ext}^{1}(\mathcal{O}_{x},\mathcal{O}_{x})=\mathbb{C}^{2} @>>> \mathrm{Ext}^{1}(\mathcal{O}_{x},A)=\mathbb{C} @>>> 0\\
@VVV @VVV @VVV @VVV\\
0 @>>> \mathrm{Ext}^{1}(B,\mathcal{O}_{x})=\mathbb{C} @>>> \mathrm{Ext}^{1}(B,A)=\mathbb{C}^{n} @>>> \mathrm{Ext}^{2}(B,B)=\mathbb{C}^{n-1}\\
@VVV @VVV @VVV @VVV\\
0 @>>> 0 @>>> \mathrm{Ext}^{2}(A,A)=\mathbb{C}^{n-1} @>>> 0
\end{CD}$
\end{center}

\bigskip

The proof of Proposition \ref{prop:soo} is the following:
\begin{proof}
Since $u$, $w$, $v_{i}$, $u+w$ and $w+v_{i}$ are coming from the tangent spaces $T_{\mathbb{P}^{n-1},E}$ and $T_{X,E}$, they are versal deformations and in particular can be lifted to the second order. By Proposition 2, we must have $\kappa_{2}(u)=u[1]\circ u=0$ and similar equation for the rest elements as well. Then it is a straightforward computation that $\kappa_{2}(\xi)=\sum_{i=1}^{n-2}ac_{i}(u[1]\circ v_{i}+v_{i}[1]\circ u)$ by using these equations.

It only remains to show that $\{u[1]\circ v_{i}+v_{i}[1]\circ u|i=1,2,\dots,n-2\}$ are linearly independent. Suppose not, then we must have a nonzero linear relation $\sum_{i=1}^{n-2}p_{i}(u[1]\circ v_{i}+v_{i}[1]\circ u)=0$. We can rewrite it as $u[1]\circ v+ v[1]\circ u=0$, where $v=\sum_{i=1}^{n}p_{i}v_{i}$ is some nonzero element in $N_{C/\mathbb{P}^{n-1},E}$. Since $N_{C/\mathbb{P}^{n-1},E}\subseteq\mathrm{ker}(\theta)=\mathrm{im}(\phi)=T_{\mathbb{P}^{n-1},E}$, we can write $v=a[1]\circ f\circ b$ for some $f\in\mathrm{Ext}^{1}(B,A)$. It is not very hard to see from the diagram in Lemma $4$ that $T_{\mathbb{P}^{n-1},E}$ can be identified with the cokernel of $d[1]\circ-\circ c:\mathrm{Ext}^{1}(\mathcal{O}_{x},\mathcal{O}_{x})\longrightarrow\mathrm{Ext}^{1}(B,A)$, which is naturally $\mathrm{Ext}^{2}(A,A)$ or $\mathrm{Ext}^{2}(B,B)$, hence $f[1]\circ e\neq 0$ in $\mathrm{Ext}^{2}(A,A)$. Moreover, from the diagram in Lemma $1$,$N_{C/\mathbb{P^{n-1}},E}$ can be identified with the cokernel of $\eta[1]\circ-:\mathrm{Ext}^{1}(A,B)\longrightarrow\mathrm{Ext}^{2}(A,A)$, which is naturally $\mathrm{Ext}^{2}(A,E)$, hence $a[2]\circ f\circ e\neq0$. On the other hand, we have
\begin{align*}
0=&(u[1]\circ v+ v[1]\circ u)\circ a\\
=&u[1]\circ a[1]\circ f\circ (b\circ a)+ a[2]\circ f[1]\circ (b[1]\circ u\circ a)\\
=&a[2]\circ f[1]\circ e,
\end{align*}
which is a contradiction. Hence $\{u[1]\circ v_{i}+v_{i}[1]\circ u|i=1,2,\dots,n-2\}$ are linearly independent.
\end{proof}

To summarise, Proposition \ref{prop:soo} tells us that $\kappa_{2}^{-1}(0)$ is cut out by equations $ac_{i}=0$, $i=1,\cdots,n-2$ with respect to the bases $u[1]\circ v_{i}+v_{i}[1]\circ u$, $i=1,\cdots,n-2$. From the discussion we had before Lemma \ref{lem:erc}, this is enough to conclude that $M_{\tau}([\mathcal{O}_{x}])$ is reduced.

\begin{thm}
	\label{thm:mod}
	The map $\gamma$ induces an isomorphism $X \sqcup_C \BP^{n-1} \rightarrow M_{\tau}([\CO_x])$, where $C$ is embedded in $\BP^{n-1}$ as a rational normal curve. 
\end{thm}

\begin{proof}
By the work of the previous section, $M_{\tau}([\CO_x])$ is reduced. Therefore we may apply Lemma \ref{lem:injred} and \cite[Corollarly 14.10]{Harris} to see that  Proposition \ref{prop:bijection} and Lemma \ref{lem:tan} show that $\gamma$ induces an isomorphism. 	
\end{proof}

\bibliographystyle{alpha}
\bibliography{research}

\begin{thebibliography}{BMS14}

\bibitem[AB13]{AB}
Daniele Arcara and Aaron Bertram.
\newblock Bridgeland-stable moduli spaces for {$K$}-trivial surfaces.
\newblock {\em J. Eur. Math. Soc. (JEMS)}, 15(1):1--38, 2013.
\newblock With an appendix by Max Lieblich.

\bibitem[AS18]{AS18}
E.~Arbarello and G.~Sacc\`a.
\newblock Singularities of moduli spaces of sheaves on {K}3 surfaces and
  {N}akajima quiver varieties.
\newblock {\em Adv. Math.}, 329:649--703, 2018.

\bibitem[BM11]{BM2011}
Arend Bayer and Emanuele Macr{\`{\i}}.
\newblock The space of stability conditions on the local projective plane.
\newblock {\em Duke Math. J.}, 160(2):263--322, 2011.

\bibitem[BMS14]{BMS}
Arend Bayer, Emanuele Macr{\`{\i}}, and Paolo Stellari.
\newblock The space of stability conditions on abelian threefolds, and on some
  {C}alabi-{Y}au threefolds.
\newblock 2014.

\bibitem[Bog78]{BG}
F.A. Bogomolov.
\newblock Holomorphic tensors and vector bundles on projective manifolds.
\newblock {\em Izv. Akad. Nauk SSSR Ser. Mat.}, 42(6):1227--1287, 1978.

\bibitem[Bri07]{Bridgeland}
Tom Bridgeland.
\newblock Stability conditions on triangulated categories.
\newblock {\em Ann. of Math. (2)}, 166(2):317--345, 2007.

\bibitem[Bri08]{K3}
Tom Bridgeland.
\newblock Stability conditions on {K}3 surfaces.
\newblock {\em Duke Math. J.}, 141(2):241--291, 2008.

\bibitem[Gie79]{BG2}
D.~Gieseker.
\newblock On a theorem of {B}ogomolov on {C}hern classes of stable bundles.
\newblock {\em Amer. J. Math.}, 101(1):77--85, 1979.

\bibitem[Har80]{Harris}
Joe Harris.
\newblock The genus of space curves.
\newblock {\em Math. Ann.}, 249(3):191--204, 1980.

\bibitem[KLS06]{DMCh06}
D.~Kaledin, M.~Lehn, and Ch. Sorger.
\newblock Singular symplectic moduli spaces.
\newblock {\em Invent. Math.}, 164(3):591--614, 2006.

\bibitem[KS90]{KaSch}
M.~Kashiwara and P.~Schapira.
\newblock {\em Sheaves on manifolds}.
\newblock Springer, 1990.

\bibitem[KS08]{KS}
Maxim Konstevich and Yan Soibelman.
\newblock Stability structures, motivic {D}onaldson-{T}homas invariants and
  cluster transformations.
\newblock 2008.

\bibitem[Lie06]{Lie06}
Max Lieblich.
\newblock Moduli of complexes on a proper morphism.
\newblock {\em J. Algebraic Geom.}, 15(1):175--206, 2006.

\bibitem[LP97]{LP}
J.~Le~Poitier.
\newblock {\em Lectures on Vector Bundles}.
\newblock Cambridge Studies in Advanced Mathematics, 1997.

\bibitem[LS06]{LS06}
Manfred Lehn and Christoph Sorger.
\newblock La singularit\'e de {O}'{G}rady.
\newblock {\em J. Algebraic Geom.}, 15(4):753--770, 2006.

\bibitem[Tod13]{Toda13}
Yukinobu Toda.
\newblock Stability conditions and extremal contractions.
\newblock {\em Math. Ann.}, 357(2):631--685, 2013.

\bibitem[Tod14]{Toda14}
Yukinobu Toda.
\newblock Stability conditions and birational geometry of projective surfaces.
\newblock {\em Compositio Mathematica}, 2014.

\bibitem[VdB02]{vdbergh}
Michel Van~den Bergh.
\newblock Three-dimensional flops and noncommutative rings.
\newblock {\em Duke Math. J.}, 2002.

\end{thebibliography}

\end{document}